\documentclass{article}

\usepackage{arxiv}

\usepackage[utf8]{inputenc} % allow utf-8 input
\usepackage[T1]{fontenc}    % use 8-bit T1 fonts
\usepackage{hyperref}       % hyperlinks
\usepackage{url}            % simple URL typesetting
\usepackage{booktabs}       % professional-quality tables
\usepackage{amsfonts}       % blackboard math symbols
\usepackage{nicefrac}       % compact symbols for 1/2, etc.
\usepackage{microtype}      % microtypography
\usepackage{graphicx}
\usepackage[numbers]{natbib}
\usepackage{doi}

\usepackage{amssymb}
\usepackage{graphicx}
\usepackage{stmaryrd}
\usepackage{enumerate}
\usepackage{float}
\usepackage{hyperref}
\usepackage{color}
\usepackage{amsmath}
\usepackage{amsfonts}
\usepackage{amsthm}

\usepackage{placeins}

\newtheorem{theorem}{Theorem}[section]
\newtheorem{lemma}[theorem]{Lemma}
\newtheorem{corollary}[theorem]{Corollary}

\theoremstyle{definition}
\newtheorem{definition}[theorem]{Definition}

\theoremstyle{remark}
\newtheorem{remark}[theorem]{Remark}

\numberwithin{equation}{section}

\newcommand{\Kk}{\mathbb{K}}
\newcommand{\Za}{\mathbb{Z}}
\newcommand{\Na}{\mathbb{N}}
\newcommand{\aut}{{\rm Aut}}
\newcommand{\gen}{\mathfrak{g}}
\newcommand{\Fa}{\mathbb{F}}
\newcommand{\Cc}{\mathcal{C}}
\newcommand{\lcm}{{\rm lcm}}
\newcommand{\dsum}{\displaystyle\sum}
\newcommand*\mymin[1][]{\min_{#1}\,}

\title{The p-rank of curves of Fermat type}

\author{Herivelto  Borges\thanks{Avenida Trabalhador S\~ao-carlense, 400, S\~ao Carlos, CEP 13566-590, SP, Brazil}\\
	Instituto de Ci\^encias Matem\'aticas e de Computa\c{c}\~ao \\
	Universidade de S\~ao Paulo\\
	S\~ao Carlos - SP\\
	\texttt{hborges@icmc.usp.br} \\
	\And
	Cirilo Gon\c{c}alves\thanks{Avenida Ministro Olavo Drummond, 25, Arax\'a, CEP 38180-510, MG, Brazil} \\
	Departamento de Forma\c{c}\~ao Geral\\
	Centro Federal de Educa\c{c}\~ao Tecnol\'ogica de Minas Gerais\\
	Arax\'a - MG \\
	\texttt{cirilo@cefetmg.br} \\
}

\renewcommand{\shorttitle}{The $p$-rank of curves of Fermat type}

\hypersetup{
pdftitle={The p-rank of curves of Fermat type},
pdfsubject={The p-rank of curves of Fermat type},
pdfauthor={Herivelto  Borges, Cirilo Gon\c{c}alves},
pdfkeywords={Finite  fields, Hasse-Witt invariant, p-rank,  automorphism group, supersingular curves},
}

\begin{document}
\maketitle

\begin{abstract}
%	\lipsum[1]
Let $\Kk$   be an algebraically closed field  of characteristic $p>0$. A pressing problem in the theory of algebraic curves is the  determination  of the $p$-rank of  a (nonsingular, projective, irreducible) curve $\mathcal{X}$  over  $\Kk$. This birational  invariant affects  arithmetic and geometric  properties of $\mathcal{X}$, and its fundamental role in the study of the  automorphism group 	$\aut (\mathcal{X})$ has been noted by many authors in the past few decades.  In this paper, we provide an extensive study of the $p$-rank of curves of Fermat type $y^m = x^n + 1$ over $\Kk=\bar{\mathbb{F}}_p$. We determine a combinatorial formula for this invariant in the general case and show how this leads  to explicit formulas of the $p$-rank of several such curves. By way of illustration, we present explicit formulas for more than  twenty subfamilies of such curves, where $m$ and $n$ are generally  given in terms of $p$. We  also show how the approach can be used to compute the $p$-rank of other types of curves.
\end{abstract}

% keywords can be removed

\section{Introduction}
Let $\mathcal{X}$ be a nonsingular, projective, irreducible   curve of genus $\gen > 0$ over an algebraically closed field $\Kk$ of characteristic $p>0$, and let $J[p]$ be the kernel of the multiplication-by-$p$ morphism on the jacobian $J$ of $\mathcal{X}$.
%Arithmetic and geometric properties of $\mathcal{X}$  are often encoded in its birational invariants, such as the genus,  automorphism group, and  $p$-rank. The latter is t
The integer $\gamma(\mathcal{X}) \in \{0, \ldots, \gen\}$ for which  $J[p](\Kk) \cong (\Za/p\Za)^{\gamma(\mathcal{X})}$ is a birational invariant, called the $p$-rank of  $\mathcal{X}$. This number  naturally connects with other birational invariants, and its   study  is fundamental to addressing  several problems related to the classification of curves over finite fields.
%For instance, the fact that supersingular curves have zero $p$-rank \cite{OORT,SILV} is relevant in the study of curves whose number $N$ of $\Fa_q$-rational points attains the Hasse-Weil bound \cite{StichXing} 
%\begin{equation*}|N-(q+1)| \leq 2\gen\sqrt{q}.\end{equation*}

The fact that  the $p$-rank can affect the study of the automorphism group $\aut(\mathcal{X})$ of a curve $\mathcal{X}$ of genus $\gen \geq 2$ can be easily  illustrated by Nakajima's results in \cite[Theorems 1, 2, and 3]{Nakajima}. There is also a well-known connection between zero $p$-rank and large automorphism groups. Indeed, for $p>2$ and $G \leq \aut(\mathcal{X})$, with $G$ solvable or $\gen$ even, Giulietti and Korchm\'aros proved that there exists a constant $c_p>0$ such that if $|G| > c_p\gen^2$, then the $p$-rank of $\mathcal{X}$ is zero \cite{KorcGiul}. On the other hand, a conjecture by Guralnick and Zieve, reported in   \cite{Gunby}, \cite{Aristides1}, and \cite{korch-mont},    states that if $\mathcal{X}$ is ordinary, that is, $\gamma(\mathcal{X})=\gen$, then $|G| \leq c_p\gen^{8/5}$ for some constant $c_p>0$. With the additional hypothesis that $G$ is solvable,  Korchm\'aros and Montanucci proved that  $|G| \leq 34(\gen+1)^{3/2}$ \cite{korch-mont}.

Information on the  $p$-rank $\gamma(\mathcal{X})$ contributes to the study of  other important notions, such as supersingularity, zeta function, and the $a$-number $a(\mathcal{X}) \in \{0, \ldots, \gen-\gamma(\mathcal{X})\}$, which is  another birational invariant \cite{Achter-Pries1}, \cite{Pries-et-al},\cite{Pietro}, \cite{OORT}.

%\textcolor{red}{citar[DOI: 10.1093/imrn/rnn053]} \textcolor{blue}{, \cite {Achter-Pries2}}.

%{\color{red}Information about $p$-rank also contributes to the study of the $a$-number, which is another birational invariant that arises as the dimension of the vector space of exact holomorphic differentials. This relationship is well established in the literature \cite{Pries-et-al, Pietro}.}

The characterization of the $p$-rank for a few curves can be found in the literature, of which the case of elliptic curves is well known \cite{SILV}.
It is also known that the $p$-rank of a curve $\mathcal{X}$ is equal to its Hasse-Witt invariant, and then  $\gamma(\mathcal{X})$ can be investigated through the Hasse-Witt matrix of $\mathcal{X}$. In 1972, Miller studied a family of hyperelliptic curves for which this matrix is invertible \cite{Miller}.
%In 1972, Miller studied a family of hyperelliptic curves for which the Hasse-Witt matrix is invertible \cite{Miller}. 
When $\mathcal{X}$ is a Fermat or hyperelliptic curve, Yui provided conditions for which $\gamma(\mathcal{X})=0$ or $\gamma(\mathcal{X})=\gen$ \cite{Yui-hyper,Yu80}, and Kodama and Washio investigated  relationships between the ranks of certain Hasse-Witt matrices  \cite{KodamaW1,KodamaW2}. 
%{\color{red} For prime numbers $p_1 \neq p_2$ and a Fermat (or hyperelliptic) curve $\mathcal{X}_i$, defined over a perfect field of characteristic $p_i$, Kodama and Washio investigated some relationships between the ranks of the Hasse-Witt matrices of $ \mathcal{X}_1$ and $\mathcal{X}_2$ \cite{KodamaW1,KodamaW2}}. 
Some connections between  Hasse-Witt matrices and Weierstrass semigroups were explored by St\"ohr and Viana \cite{Stohr-Viana}. The $p$-rank of the Fermat curves of prime degree $l \neq p$ was studied by Gonz\'alez in \cite{Gonz}, where explicit formulas were obtained for some particular cases. The problem of the existence of hyperelliptic curves with prescribed genus and $p$-rank was solved by Zhu for $p=2$ \cite{ZHU} and by Glass and Pries for $p > 2$ \cite{GlassPries}. In 2010, Bassa and Beelen obtained the $p$-rank of the Fermat curve $\mathcal{F}_{q-1} : y^{q-1}=x^{q-1}+1$, where $q$ is a power of $p$ \cite[Theorem 19]{Bassa}.

Though  restricted to particular families of curves, the aforementioned results are  of absolute interest. For instance, using  these curves along with  the Deuring-Shafarevich formula \cite[Theorem 4.2]{Subrao}, one can determine the $p$-rank of several other curves that arise as their Galois $p$-power covers. This approach was used by Giulietti, Korchm\'aros, and Timpanella to compute the $p$-rank of the Dickson-Guralnick-Zieve curve 
\begin{equation*}\Cc: \dfrac{(x^{q^3}-x)(y^q-y)-(x^q-x)(y^{q^3}-y)}{(x^{q^2}-x)(y^q-y)-(x^q-x)(y^{q^2}-y)} = 0\end{equation*}
in terms of $\gamma(\mathcal{F}_{q-1})$ \cite[Remark 4]{KGT}. %,  and consequently $\gamma(\Cc) = q^4((p+1)/2)^h -q^4-q^3+q+1>0$, where $q=p^h$ (see \cite{Bassa} and Theorem \ref{+formulas-thm1}). 
They determined the p-rank of $\Cc$ when $q=p$, concluding that $\Cc$ is an ordinary curve with  $|\aut(\Cc)|\approx c\gen^{8/5}$, that is, $\Cc$ attains the conjectural Guralnick-Zieve bound. It is worth noting  that the Deuring-Shafarevich formula is not the only method  to obtain the $p$-rank of additional curves (see, e.g., Lemma \ref{p-rank-morphism}, Theorem \ref{p-rank-Fnm-kani-rosen}, and Corollary \ref{cor-04022020}).

In \cite[Remark 4]{KGT}, the authors mention that the $p$-rank of $\mathcal{F}_{q-1} : y^{q-1}=x^{q-1}+1$ is known only when $q=p$, although the $p$-rank of this specific Fermat curve was previously obtained in \cite[Theorem 19]{Bassa}. These facts highlight not only the interest in the $p$-rank of Fermat-type curves but also the need for further research and discussion. %

For integers $n, m \geq 2$   coprime to $p$, this work determines the p-rank of the curves $y^m = x^n + 1$ over $\Kk=\bar{\mathbb{F}}_p$. More precisely, it proves the following main results.
\begin{theorem}\label{teo-intro}
	Let $n, m\geq 2$  be integers coprime to $p$, and let $h,\alpha,\beta \in \mathbb{Z}_{>0}$ be such that $\alpha n = \beta m = p^h-1$. Then the $p$-rank of the curve $\mathcal{F}_{m,n}: y^{m} = x^{n} + 1$ is given by 
	\begin{equation*}\gamma(\mathcal{F}_{m,n})=\#T-(m+n+\gcd(m,n)),\end{equation*}
	where  
	%\begin{small}
	\begin{equation*}
	T = \left\{ (i,j) \in \Na^2 \mid  0 \leq i\alpha \leq j\beta \leq p^h-1 \; \mbox{and}\; \dbinom{j\beta}{i\alpha} \not\equiv 0 \pmod{p} \right\}.
	\end{equation*}
	%\end{small}
\end{theorem}
Note that the cardinality of $T$ can be easily calculated for small values of $n$ and $m$, and even when such values are larger, the task can be computationally simple.  
%In the generality of our hypotheses in Theorem \ref{teo-intro}, one should  not expect to provide a general formula for the cardinality of $T$; this reflects the intrinsic difficulty of the general problem of computing p-ranks. 
Beyond this, Theorem \ref{teo-intro} effectively provides  explicit formulas
for the $p$-rank of  curves of type $\mathcal{F}_{m,n}$. To illustrate its reach, we provide
Tables  1---5 % \ref{tab-Fmn} \ref{tab-Fmn1}, \ref{tab-Cn1}, \ref{tab-Cn2} and \ref{tab-Dn} below summarize a few cases we have worked out to arrive at explicit formulas.
below, summarizing a few cases.

%\FloatBarrier
\begin{table}[ht]%[width=1\textwidth,cols=4,pos=ht]
	\centering
	\caption{The $p$-rank of $y^{m_1} = x^{n_1} + 1$. Throughout the table, $d=\gcd(m,n)$.}\label{tab-Fmn}	
	\scalebox{0.76}{\begin{tabular}{cccc}%{\tblwidth}{@{}CCCC@{}}
			\toprule
			$m_1$ & $n_1$ & $\gamma(\mathcal{F}_{m_1,n_1})$ & Theorem \\ 
			\midrule
             $\begin{array}{c} m, \\ m \mid p-1 \end{array}$& $\begin{array}{c} n, \\ n \mid p+1 \end{array}$ & $2\sum_{i=\left \lceil{  m/2}\right \rceil }^{m}  \left \lfloor{\frac{  i(p-1)/m+1   }{(p+1)/n}}\right \rfloor - (m- \left \lceil{ m/2}\right \rceil +2)  (n-1)+1-d$ & \ref{p+1,p-1} \vspace{0.2cm}\\
             $\begin{array}{c} m, \\ m \mid p-1 \end{array}$& $p^n-1$ & $\sum_{i=1}^{m-1} ( i(p-1)/m + 1)^n - 2(m-1)$ & \ref{c-m,p^h-1} \vspace{0.2cm}\\
			$p^2+p+1$ & $p^2+p+1$ & $p(p+1)(p^2 + p + 2)/8$ & \ref{h=3-26072022} \vspace{0.3cm}\\
			$p^2+p+1$ & $p^3-1$ & $p(p+1)(p^3 + 2p^2 + 3p - 14)/8$ & \ref{h=3-26072022} \vspace{0.3cm}\\
			$(p^{n}-1)/2$ & $(p^{n}-1)/2$  &   $(p+1)^n(p^n+3)/2^{n+2}-3(p^n-1)/2$ & \ref{cor-12032020a} \vspace{0.3cm}\\
			$(p^{n}-1)/2$ & $p^{n}-1$  &   $(p+1)^n(p^n+1)/2^{n+1} - 2(p^n-1)$ & \ref{cor-12032020a} \vspace{0.3cm}\\
			$p^m-1$ & $p^n-1$ & $\Big(\sum_{i=0}^{p-1} \big((i+1)^{\frac{m}{d}}-i^{\frac{m}{d}}\big)(p-i)^{\frac{n}{d}}\Big)^d - (p^m+p^n+p^d-3)$ & \ref{+formulas-thm1} \vspace{0.3cm}\\	
			$\begin{array}{c} p^m+1, \\ p \neq 2 \end{array}$ & $\begin{array}{c} p^n-1, \\ n/d\;\mbox{even} \end{array}$  &  $\Big(\sum_{i=0}^{p-1} \big((p-i) (i+1)\big)^{\frac{n}{2d}} + \hspace{-0.3cm} \sum\limits_{1 \leq j \leq i \leq p-1} \hspace{-0.3cm} \big((j+1)^{\frac{m}{d}}-2j^{\frac{m}{d}}+(j-1)^{\frac{m}{d}}\big)\big((p-i) (i-j+1)\big)^{\frac{n}{2d}}  \Big)^d - (p^{m}+p^d)$ & \ref{+formulas-thm3} \vspace{0.3cm}\\	 
			$\begin{array}{c} p^m+1, \\ p \neq 2 \end{array}$ & $\begin{array}{c} p^n-1, \\ n/d\;\mbox{odd} \end{array}$ & $\dfrac{1}{2^n}\Big((p+1)^{\frac{n}{d}} + \sum_{i=1}^{(p-1)/2} \big((2i+1)^{\frac{m}{d}}-(2i-1)^{\frac{m}{d}}\big)(p-2i+1)^{\frac{n}{d}}\Big)^d - (p^m+1)$ & \ref{+formulas-thm2} \\ 
			\bottomrule
	\end{tabular}}
\end{table}

%\FloatBarrier

%\FloatBarrier

\begin{table}[ht]%[width=.8\textwidth,cols=4,pos=ht]
	\centering
	\caption{The $p$-rank of $y^{m_1} = x^{n_1} + 1$}\label{tab-Fmn1}
	\begin{tabular}{cccc}%{\tblwidth}{@{}CCCC@{}}
		\toprule
		$m_1$ & $n_1$ & $\gamma(\mathcal{F}_{m_1,n_1})$ &  \\ 
		\midrule
		$\begin{array}{c} m_1 \mid 2^{m}-1, \\ p=2 \end{array}$ & $\begin{array}{c} n_1 \mid 2^{n}-1, \\ \gcd(n,m)=1  \end{array}$ & $0$ & Corollary \ref{+formulas-thm1-cor1} \vspace{0.3cm}\\ 
		$m_1 \mid p^{m}+1$ & $n_1 \mid p^{n}+1$ & $0$ & Theorem \ref{+formulas-thm4} \vspace{0.3cm}\\  
		$\begin{array}{c} m_1 \mid 2^{m}+1,\\ p=2 \end{array}$ & $\begin{array}{c} n_1 \mid 2^{n}-1, \\ n/\gcd(m,n)\;\mbox{odd} \end{array}$ & $0$ & Corollary \ref{16112022-1} \\ 
		\bottomrule
	\end{tabular}
\end{table}

%\FloatBarrier

\begin{table}[ht]%[width=.6\textwidth,cols=2,pos=ht]
	\centering
	\caption{The $p$-rank of $y^2=x^n+1$ -- Corollary \ref{corollary-p-rank-c1a}.}\label{tab-Cn1}
	\begin{tabular}{cc}%{\tblwidth}{@{}CC@{}}
		\toprule
		$n$ &  $\gamma(\mathcal{F}_{2,n})$  \\ 
		\midrule                             
		%$p^r-1$ &  $\Big( (p+1)/2\Big)^r-2$  \vspace{0.1cm}\\
		$2(p^r-1)$ &  $\Big( (p+1)/2\Big)^r-2$ \vspace{0.1cm}\\                         
		$2(p^r+1)$ &  $\Big( (p+1)/2\Big)^r$ \vspace{0.1cm}\\
		$p^{2r}+p^r+1$ &   $\Big( (p+3)(p+1)/8\Big)^r$  \vspace{0.1cm}\\  
		$2(p^{2r}+p^r+1)$ &   $2\Big( (p+3)(p+1)/8\Big)^r$  \vspace{0.1cm}\\   
		$p^{3r}-p^{2r}+p^r-1$ &  $\Big( (p^2+2p+3)(p+1)/12\Big)^r-2$  \\
		\bottomrule 
	\end{tabular}
\end{table}

\FloatBarrier

Despite the vanishing of their $p$-ranks, note that most curves in Table \ref{tab-Fmn1} are not supersingular, see Remark \ref{supersing}.
To complement Table \ref{tab-Cn1}, let us mention that in general  the $p$-rank of $\mathcal{F}_{2,n}: y^{2} = x^{n} + 1$ is given by
\begin{equation}\label{F2n}
\gamma(\mathcal{F}_{2,n}) = \dfrac{n}{p^h-1}  \Big(((p+1)/2)^h - \delta_n \Big),\end{equation} 
for a suitable  integer $\delta_n \leq ((p+1)/2)^h$ dependent on $\alpha =(p^h-1)/n \in \Za_{>0}$ (Theorem \ref{p-rank-28-06-2022}). Table \ref{tab-Cn2} presents some values of $\delta_n$.

%\FloatBarrier
\begin{table}[ht]%[width=.5\textwidth,cols=3,pos=ht]
	\centering
	\caption{$\delta_n$ for $y^2=x^n+1$ -- Theorem \ref{p-rank-28-06-2022}.}\label{tab-Cn2}
	\begin{tabular}{ccc}%{\tblwidth}{@{}CCC@{}}
		\toprule
		& \multicolumn{2}{c}{$\delta_n$}\\
		\cline{2 - 3} % linha horizontal entre as colunas
		% 2 e 4
		hypotheses & $n$ odd & $n$ even \\
		\midrule
		$\alpha \mid \frac{p-1}{2}$  & not applicable & $\alpha+1$   \vspace{0.3cm}\\ 
		$\alpha \mid \frac{p+1}{2}$  & $\alpha$ & $2\alpha$  \vspace{0.3cm}\\    
		$\begin{array}{c}  \alpha \mid  \frac{p+3}{2}, \\ (p-1) \nmid n  \end{array}$  & $\alpha+(-1)^h$ & $2\alpha+(-1)^h$  \vspace{0.3cm}\\  
		$\begin{array}{c}  \alpha \mid  \frac{p+3}{2}, \\ (p-1) \mid n\end{array}$  & $\alpha+(-1)^h(1-\alpha)$ & $2\alpha+(-1)^h(1-\alpha)$  \vspace{0.3cm} \\    
		$\alpha \mid p-1$  & $(\alpha+2)/2$ &   \\
		\bottomrule  
	\end{tabular}
\end{table} 
%\FloatBarrier
The curve $\mathcal{F}_{2,n}:y^2=x^{r\frac{p^{h}-1}{p-1}}+1$, where $r$ is an even divisor of $p-1$, typifies the first line of Table \ref{tab-Cn2}. In this case, we have $\delta_n=\alpha+1=(p-1) / r+1$, and then \eqref{F2n} gives
 $$\gamma(\mathcal{F}_{2,n})=\frac{r}{p-1}\left(((p+1)/2)^h-1\right)-\frac{r}{p-1}.$$
 In particular,  the curve $y^2=x^{p^{h}-1}+1$ has $p$-rank  $((p+1)/2)^r-2$, a fact recently proved by Pries and Ulmer in \cite[Proposition 10.3]{PRIES-ULMER}. Note that this  also follows from  the second line of Table \ref{tab-Fmn}.

As a by-product of the previous results, one can obtain the  $p$-rank of  other types of curves. This can be illustrated   by  the following curves $$\mathcal{D}_n: y^2=x(x^{n}+1).$$ If $n=(p^h-1)/\alpha \in \Za_{>0}$ is odd, then $\gamma(\mathcal{D}_n) = \gamma(\mathcal{C}_n)$ (Corollary \ref{cor-04022020}). 
Otherwise, from  Kani and Rosen's result in Theorem \ref{teo-13032020}, it follows that  $\gamma(\mathcal{D}_n)=0$ when $\alpha$ is odd, and
\begin{equation*}\gamma(\mathcal{D}_n) = \gamma(\mathcal{F}_{2,2n}) - \gamma(\mathcal{F}_{2,n}) = \dfrac{n}{p^h-1} \Big(( (p+1)/2 )^h - \tilde{\delta}_n\Big),\end{equation*} 
with $\tilde{\delta}_n = 2\delta_{2n} - \delta_{n}$,  when $\alpha$ is even (Theorem \ref{tab5-24092022}). 
Some values of $\tilde{\delta}_n$ are displayed in Table \ref{tab-Dn}.  

%\FloatBarrier  

\begin{table}[ht]%[width=.4\textwidth,cols=2,pos=ht]
	\centering
	\caption{$\tilde{\delta}_n$ for $y^2=x(x^n+1)$ -- Theorem \ref{tab5-24092022}.}\label{tab-Dn}
	\begin{tabular}{cc}%{\tblwidth}{@{}CC@{}}
		\toprule
		hypotheses &  $\tilde{\delta}_n$ \\
		\midrule
		$\alpha \mid \frac{p-1}{2}$    & $1$  \vspace{0.3cm}\\ 
		$\alpha \mid \frac{p+1}{2}$    & $0$   \vspace{0.3cm}\\  
		$\begin{array}{c}  \alpha \mid  \frac{p+3}{2}, \;\mbox{and} \\ (p-1) \nmid 2n \;\mbox{or}\\ (p-1) \mid n  \end{array}$  & $(-1)^h$  \vspace{0.3cm}\\   
		$\begin{array}{c}  \alpha \mid  \frac{p+3}{2}, \\ (p-1) \mid 2n \;\mbox{and}\\ (p-1) \nmid n  \end{array}$   & $(-1)^h(1-\alpha)$   \\
		\bottomrule   
	\end{tabular}
\end{table} 

%\FloatBarrier

This paper is organized as follows. In Section \ref{sec1}, we establish notation and collect some preliminary results. In Section \ref{sec2}, we prove the main result. In Sections \ref{sec3} and \ref{sec4}, we use the results of Section \ref{sec2} to present closed formulas for the $p$-rank of several Fermat-type curves. 
Finally in Section \ref{sec5}, we illustrate how the results of the previous sections can be used to obtain  the $p$-rank of other types of curves.

\section*{Notation}

The following notation is used throughout this text.
\begin{itemize}
	\item $\Fa_p$ is the prime field of characteristic $p>0$
	\item $\Kk=\bar{\mathbb{F}}_p$   is the algebraic closure of $\Fa_p$
	%\item $\Fa_q$ is the finite field with $q=p^h$ elements
	%\item $J_{\mathcal{X}}$ is the jacobian of the curve $\mathcal{X}$
	\item $\mathcal{F}_{m,n}$ is the curve $y^m=x^n+1$ over $\Kk$, where $m$ and $n$ are not divisible by $p$
	\item $\mathcal{C}_{n}$ is the curve $\mathcal{F}_{2,n}$
	\item $\mathcal{D}_{n}$ is the curve $y^2=x(x^n+1)$  over $\Kk$, where $p>2$  does not divide $n$.
	%	\item $\mathcal{D}_{n}$ is the curve $y^2=x^n+x$ defined over $\Fa_q$, where $q$ is odd and $n-1$ is a positive integer not divisible by $p$
	\item $\gamma(\mathcal{C})$ is the $p$-rank of the curve $\mathcal{C}$
	\item $\mathbb{N}=\{0,1,2,\ldots\}$ is the set of natural numbers
	\item $\Za_{>0}$ is the set of positive integers
	\item $\lfloor x\rfloor=\max \{m \in \mathbb{Z} \mid m \leq x\}$ and $\lceil x\rceil=\min \{m \in \mathbb{Z} \mid m \geq x\}$, where $x$ is a real number
	\item $\llbracket r, s \rrbracket$, where  $r\leq s$ are integers,   is the set $\{r,\ldots,s \}$
	\item  $\llbracket r, s \rrbracket^h$, with $h\in \Za_{>0}$,  is the cartesian product $\llbracket r, s \rrbracket\times \cdots \times  \llbracket r, s \rrbracket\ $ ($h$ times).
\end{itemize}

\section{Preliminaries}\label{sec1}

Let $\mathcal{X}$ be a projective, nonsingular, and irreducible algebraic curve of genus $\gen > 0$ defined over $\Kk$. Let $\Kk(\mathcal{X})=\Kk(x,y)$ be the function field of $\mathcal{X}$, where $x \in \Kk(\mathcal{X})$ is a separating variable. 
Let $\Delta_{\mathcal{X}} = \{ \varphi dx : \varphi \in \Kk(\mathcal{X})\}$ be the module of differentials of $\Kk(\mathcal{X})$. For any $\omega \in \Delta_{\mathcal{X}}$, we say that

\begin{itemize}
	\item $\omega$ is holomorphic if $\rm{div}(\omega)$ is effective
	
	\item $\omega$ is exact if $\omega = df$ for some $f \in \Kk(\mathcal{X})$
	
	\item $\omega$ is logarithmic if $\omega = df/f$ for some $f \neq 0$ in $\Kk(\mathcal{X})$.
\end{itemize}

The Cartier operator $C :  \Delta_{\mathcal{X}} \to \Delta_{\mathcal{X}}$ is defined as follows. For each $\omega \in \Delta_{\mathcal{X}}$, let $f_0, \ldots, f_{p-1} \in \Kk(\mathcal{X})$ be such that $\omega = (f_0^p + f_1^px + \cdots + f_{p-1}^px^{p-1})dx$, and set
\begin{equation*}%\label{cartier-operator-definition}
C(\omega) := f_{p-1} dx.
\end{equation*}
This operator has the following properties:
\begin{enumerate}[\rm (i)]
	
	\item The value of $C(\omega)$ does not depend on the choice of $x$
	
	\item $C$ is $1/p$-linear, that is, $C$ is additive and $C(f^p \omega) = f C(\omega)$ for all $f \in \Kk(\mathcal{X})$
	
	\item $C(f^{p-1}df) = df$ for all $f \in \Kk(\mathcal{X})$
	
	\item $C(\omega)=0$ if and only if $\omega$ is exact
	
	\item $C(\omega) = \omega$ if and only if $\omega$ is logarithmic
	
	\item If $\omega$ is holomorphic, then so is  $C(\omega)$.
	
\end{enumerate}
For  $n, i \in \Za_{>0}$, using the base-$p$ expansion of $i$ and the previous properties, one can easily check that
\begin{equation}\label{Theo1}
C^n (x^i dx) = \left\{
\begin{array}{ll}
0 & \text{if\;} i \not\equiv -1 \pmod{p^n} \\
x^{s-1} dx & \text{if\;} i+1 = sp^n,
\end{array}
\right.
\end{equation}
where  $C^n = C \circ \cdots \circ C$ ($n$ times). For further details, see \cite{GarSaeed}, \cite{KORCH}, and \cite[\textsection 10 and \textsection 11]{Serre}.

Let $\Delta^{(1)}_{\mathcal{X}}$ be the $\gen$-dimensional $\Kk$-vector space of the holomorphic differentials on $\mathcal{X}$, and let  $\Delta^{\textrm{ log}}_{\mathcal{X}}$ be the subspace of $\Delta^{(1)}_{\mathcal{X}}$ spanned by the logarithmic differentials.

\begin{definition}\label{definition-hasse-witt-invariant-270120}
	The dimension $\gamma(\mathcal{X})$ of $\Delta^{\textrm{ log}}_{\mathcal{X}}$ is the \textit{$p$-rank} of $\mathcal{X}$. 
\end{definition}

Let  $\Delta^{0}_{\mathcal{X}}$ be the subspace  of the differentials $\omega \in \Delta^{(1)}_{\mathcal{X}}$ for which $C^n(\omega)=0$ for some integer $n>0$. 
Since the Cartier operator is $1/p$-linear, the following canonical decomposition  holds \cite{Hasse-Witt}.
\begin{equation}\label{decoposicao-espaco-diferenciais-holomorfas}
\Delta^{(1)}_{\mathcal{X}} = \Delta^{\textrm{ log}}_{\mathcal{X}} \oplus \Delta^{0}_{\mathcal{X}}.
\end{equation}

\begin{theorem}\label{prop-2201}
	Let $\omega \in \Delta^{(1)}_{\mathcal{X}}$ be a nonzero holomorphic differential. If $h \in \Za_{>0}$ is such that $C^h(\omega)=\alpha \omega$ for some $\alpha \in \Kk$,  then
	$\omega \in \Delta^{\textrm{ log}}_{\mathcal{X}}$ if and only if $\alpha \neq 0$. 
	In particular, for any basis $B = \{\omega_1, \ldots, \omega_{\gen}\}$ of $\Delta^{(1)}_{\mathcal{X}}$, where $C^h(\omega_i) = \alpha_i \omega_i$, $\alpha_i \in \Kk$, we have
	\begin{equation*}\gamma(\mathcal{X}) = \#\{\omega_i \in B \mid   i \in \llbracket1, \gen\rrbracket \mbox{\;and\;} \alpha_i \neq 0\}.\end{equation*}
\end{theorem}
\begin{proof} Suppose $\alpha \neq 0$. The decomposition \eqref{decoposicao-espaco-diferenciais-holomorfas} gives $\omega = u+v$, with $u \in \Delta^{\textrm{ log}}_{\mathcal{X}}$ and $v \in \Delta^{0}_{\mathcal{X}}$. Thus  $C^{ht}(v)=0$ for some integer $t>0$, and  $C(\Delta^{\textrm{ log}}_{\mathcal{X}})\subseteq \Delta^{\textrm{ log}}_{\mathcal{X}}$    implies  $C^{ht}(u) \in  \Delta^{\textrm{ log}}_{\mathcal{X}}$. Since $C^h(\omega)=\alpha \omega$, with  $\alpha\neq 0$, the $1/p$-linearity of $C$  yields $C^{ht}(w)=\beta w=\beta(u+v)$ for some $\beta \neq 0$. Hence 
	\begin{equation*}
	\beta v=C^{ht}(w) -\beta u= C^{ht}(u+v) -\beta u =C^{ht}(u) -\beta u \in \Delta^{\textrm{ log}}_{\mathcal{X}}
	\end{equation*}
	gives $v = 0$, and then $\omega = u \in \Delta^{\textrm{ log}}_{\mathcal{X}}$. The converse is clear.
	For the final assertion, note that the basis $B$ for $\Delta^{(1)}_{\mathcal{X}} = \Delta^{\textrm{ log}}_{\mathcal{X}} \oplus \Delta^{0}_{\mathcal{X}}$ is such that 
	\begin{equation*}B_1 = \{\omega_i \in B \mid  i \in \llbracket1, \gen\rrbracket \mbox{\;and\;} \alpha_i \neq 0\} \subseteq \Delta^{\textrm{ log}}_{\mathcal{X}}\end{equation*} 
	and $B \setminus B_1 \subseteq \Delta^{0}_{\mathcal{X}}$. Therefore,  $B_1$ is a basis of $\Delta^{\textrm{ log}}_{\mathcal{X}}$, and then $\gamma(\mathcal{X})=\#B_1$.
\end{proof}

Let us recall that the $p$-rank  $\gamma (J)$ of a jacobian $J$, more generally, of an abelian variety,  is invariant under isogeny, and  $\gamma(J_1 \times J_2) = \gamma(J_1) + \gamma(J_2)$ for jacobians  $J_1$, $J_2$ defined over $\Kk$ \cite{Geer,Mum}. In particular, the following result  is suitable   to investigate the $p$-rank of a curve $\mathcal{X}$.

\begin{theorem}[\cite{KaniRosen}, Theorem B]\label{teo-13032020}
	Let $G \leq \aut(\mathcal{X})$ be a (finite) subgroup such that $G = H_1 \cup \ldots \cup H_t$, where the subgroups $H_i \leq G$ satisfy $H_i \cap H_j = \{1\}$ for $i \neq j$. Then we have the isogeny relation
	\begin{equation*}J_{\mathcal{X}}^{t-1} \times J_{\mathcal{X}/G}^g \sim J_{\mathcal{X}/H_1}^{h_1} \times \cdots \times J_{\mathcal{X}/H_t}^{h_t},\end{equation*}
	where $g = |G|$, $h_i = |H_i|$, and $J^n = J \times \cdots \times J$ ($n$ times).
\end{theorem}

The following relates the $p$-rank of a  curve to the $p$-rank of any of its (sub)covers.

\begin{lemma} [\cite{Bassa}, Lemma 6]\label{p-rank-morphism} Let $\mathcal{C}$ and $\mathcal{D}$ be curves of genus $\gen (\mathcal{C})$  and $\gen (\mathcal{D})$, respectively, defined over a finite  field. If $f: \mathcal{C} \longrightarrow \mathcal{D}$ is a nonconstant morphism, then 
	
	\begin{enumerate}[\rm(i)]
		\item  $\gamma(\mathcal{D}) \leq \gamma(\mathcal{C}) \leq \gamma(\mathcal{D})+\gen (\mathcal{C})-\gen (\mathcal{D})$
		\item if $\gamma(\mathcal{C})=0$, then $\gamma(\mathcal{D})=0$
		\item if $\mathcal{C}$ is ordinary, then so is $\mathcal{D}$.
	\end{enumerate}

\end{lemma}

The following well-known facts will be useful in obtaining formulas for the $p$-rank of curves of  type $y^{m} = x^{n} + 1$.

\begin{lemma}[\cite{Fini,Lucas}, Lucas' Lemma]\label{lucas' theorem}
	Let $p$ be a prime, and let $a_0, b_0, \ldots, a_r, b_r \in \llbracket0, p-1\rrbracket$. Then
	\begin{equation*}
	\dbinom{a_rp^r + \cdots + a_1p + a_0}{b_rp^r + \cdots + b_1p + b_0} \equiv \prod_{i=0}^{r} \dbinom{a_i}{b_i} \pmod{p}.
	\end{equation*}
\end{lemma}

\begin{lemma}[\cite{Cou}, Lemma 2.6; \cite{McEliece}, p. 175]\label{04112022-1}
	Let $p$ be a prime, and $m, n \in \Za_{>0}$. If $d=\gcd(m,n)$, then
	\begin{equation*}\gcd(p^m+1,p^n-1) = \left\{
	\begin{array}{ll}
	1 & \mbox{\;if\;}  n/d  \mbox{\;is odd and\;} p=2 \vspace{0.2cm} \\
	2 & \mbox{\;if\;}  n/d  \mbox{\;is odd and\;} p>2 \vspace{0.2cm}\\
	p^d+1 & \mbox{\;if\;}  n/d  \mbox{\;is even\;}.
	\end{array}
	\right.\end{equation*}
\end{lemma}

A curve $\mathcal{X}$  defined over a finite field $k$ is called supersingular if  it is Hasse-Weil minimal over some extension of $k$, and  it is  well known  that supersingular curves have zero $p$-rank \cite{OORT,SILV}. Though  not affecting  our results, it is worth noting that \cite[Theorem 1.2]{Oliveira}, with some simple adjustments, yields the following (see also  \cite[Theorem 5]{SaeedTor1}).

\begin{theorem}\label{Oliva}
Let $n,m \in \mathbb{Z}_{>0}$ be such that  $\mathcal{F}_{m,n}: y^{m} = x^{n} + 1$  is an irreducible curve   over $\bar{\mathbb{F}}_p$ of genus $\gen >0$. Then $\mathcal{F}_{m,n}$ is supersingular if and only if there exists $h \in \mathbb{Z}_{>0}$ such that $n$ and $m$ divide  $p^h+1$.
\end{theorem}

\begin{remark}\label{supersing}
It follows from Lemma \ref{04112022-1} and Theorem \ref{Oliva} that the curves $y^{p^m+1} = x^{p^n+1} + 1$
are supersingular if and only if $\nu_2(n) = \nu_2(m)$, where  $\nu_2: \mathbb{Z} \rightarrow \mathbb{N} \cup\{\infty\}$ is the $2$-adic valuation. %In the same vein, one sees  that curves  of type  $y^{p^m-1} = x^{p^n-1} + 1$  cannot be supersingular.
In the same vein, one sees  that the only nonrational  curve  of type  $y^{p^m \pm 1} = x^{p^n-1} + 1$
that is supersingular  is $y^{3^m + 1} = x^{2} + 1$.
\end{remark}

\section{The p-rank of Fermat-type curves  }\label{sec2}

Throughout this section $n, m \geq 2$ are integers coprime to $p$,  $\mathcal{F}_{m,n}: y^{m} = x^{n} + 1$ is a curve over $\Kk$, and  $q=p^h$ is such that $n$ and $m$ are divisors of $q-1$. Additionally, let ${\Kk}(\mathcal{F}_{m,n})={\Kk}(x,y)$ be the  function field of $\mathcal{F}_{m,n}$.
The genus of $\mathcal{F}_{m,n}$ is given by
\begin{equation*}\label{genus-Cnm}
\gen = \dfrac{(m-1)(n-1)+1-\gcd(m,n)}{2},
%\gen = \dfrac{nm-n-m-\gcd(n,m)}{2}+1,
\end{equation*}
and the set
\begin{equation*}%\label{basis-Cnm}
B = \left\{ \omega_{i,j} = \dfrac{x^{i-1}}{y^{j}}dx \mid  (i,j) \in \Na^2 \mbox{\;with\;} m \leq im < jn \leq n(m-1)  \right\}
\end{equation*}
is a basis for the vector space of  holomorphic differentials of $\mathcal{F}_{m,n}$.
For further details on the basic properties of the curve $\mathcal{F}_{m,n}$, see \cite{Aristides} and \cite{Towse}.

Note that for $(i,j) \in \Na^2$ and $m \leq im < jn \leq n(m-1)$, the equation $y^{m} = x^{n} + 1$ yields
\begin{equation*}
\dfrac{x^{i-1}}{y^{j}} =  \dfrac{x^{i-1}}{y^{jq}} (x^n+1)^{j(q-1)/m} = \sum_{r=0}^{\frac{j(q-1)}{m}} \dbinom{\frac{j(q-1)}{m}}{r} \dfrac{x^{rn+i-1}}{y^{jq}}.
\end{equation*}
That is, the elements of $B$ can be written as 
\begin{equation}\label{eq2-Cnm}
\omega_{i,j} =  \sum_{r=0}^{\frac{j(q-1)}{m}} \dbinom{\frac{j(q-1)}{m}}{r} \dfrac{x^{rn+i-1}}{y^{jq}} dx.
\end{equation}
We now restate and prove our main result for the $p$-rank of Fermat-type curves.
\begin{theorem}\label{cor-p-rank-Cnm}
	Let $n, m\geq 2$  be integers coprime to $p$, and let $h,\alpha,\beta \in \mathbb{Z}_{>0}$ be such that $\alpha n = \beta m = p^h-1$. Then the $p$-rank of the curve $\mathcal{F}_{m,n}: y^{m} = x^{n} + 1$ is given by 
	\begin{equation*}\gamma(\mathcal{F}_{m,n})=\#T-(m+n+\gcd(m,n)),\end{equation*}
	where
	\begin{equation*}
	T = \left\{ (i,j) \in \Na^2 \mid  0 \leq i\alpha \leq j\beta \leq p^h-1 \; \mbox{and}\; \dbinom{j\beta}{i\alpha} \not\equiv 0 \pmod{p} \right\}.
	\end{equation*}
\end{theorem}
\begin{proof}
	Let $q=p^h$. The $1/p$-linearity of the Cartier operator and \eqref{eq2-Cnm} imply
	\begin{equation}\label{eq3-Cnm}
	C^h(\omega_{i,j}) = \sum_{r=0}^{\frac{j(q-1)}{m}}   \dbinom{\frac{j(q-1)}{m}}{r} \dfrac{ C^h  (x^{rn+i-1} dx)}{y^j}.
	\end{equation}
	It follows from \eqref{Theo1} that $ C^h ( x^{rn+i-1} dx)$ is nonzero if and only if  $nr+i \equiv 0 \pmod{q}$. 
	%Since $im + nj \leq 2\g-2$, %(cf. Equation \eqref{genus-Cnm} and \eqref{basis-Cnm-I}),
	%we have that $i \leq n-2$ %. Since $0 \leq r \leq q-1$, it follows that 
	Note that $r\leq j(q-1)/m$ and  $m \leq im < jn \leq n(m-1)$ entail  $1 \leq nr+i \leq nq - 1$.
	Thus $nr+i \equiv 0 \pmod{q}$ if and only if $nr+i \in \{q, 2q, \ldots, (n-1)q\}$. The latter is equivalent to  
	\begin{equation*} r \in \Big\{\frac{q-i}{n}, \frac{2q-i}{n}, \ldots, \frac{(n-1)q - i}{n}  \Big\} \cap \Za,\end{equation*}
	and thus 
	$ C^h ( x^{rn+i-1} dx)$ is nonzero if and only if $r=\frac{i(q-1)}{n}$. Therefore,
	equation (\ref{eq3-Cnm}) reduces to
	\begin{eqnarray*}
		C^h(\omega_{i,j})
		& = &   \dbinom{\frac{j(q-1)}{m}}{\frac{i(q-1)}{n}} \omega_{i,j},
	\end{eqnarray*}
	and Theorem \ref{prop-2201} yields 
	\begin{equation}\label{eq1-13042021}
	\gamma(\mathcal{F}_{m,n}) = \#\left\{ (i,j) \in \Na^2 \mid  m \leq im < jn \leq n(m-1) \; \mbox{and}\; \dbinom{j\beta}{i\alpha} \not\equiv 0 \pmod{p} \right\}.
	\end{equation}
	Note that the  pairs $(i,j) \in T$ not occurring  in   \eqref{eq1-13042021} are those for which $j = m$, $i = 0$, or $im = jn$. One can check that the number of such pairs   is $m+n+\gcd(n,m)$, and this concludes the proof.
\end{proof}

\begin{remark}
	In Theorem \ref{cor-p-rank-Cnm}, note that one can always define $h=\lcm(h_1,h_2)$, where $h_1$ and $h_2$ are the orders of $p$ in $(\Za/m\Za)^{\times}$ and $(\Za/n\Za)^{\times}$, respectively. Then $\alpha = (p^h-1)/n$ and 
	$\beta  = (p^h-1)/m$.
\end{remark}

The subsequent particular cases of Theorem \ref{cor-p-rank-Cnm} are worth highlighting. 
\begin{corollary}\label{p-rank-Fn-Fermat} 
	Let $n \in \Za_{>0}$, with $\gcd(n,p)=1$, and let $q=p^h$ be such that $\alpha = (q-1)/n \in \Za_{>0}$. Then the $p$-rank of the Fermat curve $\mathcal{F}_n: y^{n} = x^{n} + 1$ is given by \begin{equation*}\gamma(\mathcal{F}_n)=\#T-3n,\end{equation*}
	where
	\begin{equation*}
	T = \left\{ (i,j) \in \Na^2 \mid  0 \leq i \leq j \leq n \; \mbox{and}\; \dbinom{j\alpha}{i\alpha} \not\equiv 0 \pmod{p} \right\}.
	\end{equation*}
\end{corollary}

\begin{corollary}\label{p-rank-c1a}
	Let $n \in \Za_{>0}$, with $\gcd(n,p)=1$, and let $q=p^h$ be such that $\alpha = (q-1)/n \in \Za_{>0}$. Then the $p$-rank of $\Cc_n:y^2=x^n+1$ is given by 
	\begin{equation*}\gamma(\Cc_n)= \left\{\begin{array}{ll}
	\#S - 1 & \mbox{\;if\;}  n  \mbox{\;is odd\;} \vspace{0.2cm} \\
	\#S - 2 & \mbox{\;if\;}  n  \mbox{\;is even\;},
	\end{array}	\right.\end{equation*}
	where
	\begin{equation*}S = \left\{i \in \llbracket0, \lfloor n/2\rfloor\rrbracket \mid  \dbinom{\frac{q-1}{2}}{i\alpha} \not\equiv 0 \pmod{p} \right\}.\end{equation*}
	
\end{corollary}

\begin{proof}
	From  Theorem \ref{cor-p-rank-Cnm}, the problem reduces to determining  the cardinality  of
	\begin{equation*}
	T = \left\{ (i,j) \in \Na^2 \mid  0 \leq i\alpha \leq \frac{j(q-1)}{2} \leq q-1 \; \mbox{and}\; \dbinom{\frac{j(q-1)}{2}}{i\alpha} \not\equiv 0 \pmod{p} \right\}.
	\end{equation*} 
	Note that  $(i,j) \in T$ implies  $j \in \{0,1,2\}$, and $\#\{ (i,j) \in T: j \in \{0,2\}\} = n+2$. In addition, since $(q-1)/(2\alpha) = n/2$,  we have that $(i,1) \in T$ is equivalent to  $i \in \llbracket0,\lfloor n/2 \rfloor\rrbracket$ and $\dbinom{\frac{q-1}{2}}{i\alpha} \not\equiv 0 \pmod{p}$. Therefore, $\#T=\#S+n+2$, and then 
	\begin{equation*}\gamma(\Cc_n)= \#T-(n+2+\gcd(2,n)) = \left\{\begin{array}{ll}
	\#S - 1 & \mbox{\;if\;}  n  \mbox{\;is odd\;} \vspace{0.2cm} \\
	\#S - 2 & \mbox{\;if\;}  n  \mbox{\;is even\;}.
	\end{array}	\right.\end{equation*}
\end{proof}

In what follows, we use  Theorem \ref{teo-13032020} to relate the $p$-ranks of the curves $\mathcal{F}_{m,n}$ and $y^m = x^r(x^{m^{k-1}u} + 1)$, with $r \in \llbracket0, m-1\rrbracket$.

\begin{theorem}\label{p-rank-Fnm-kani-rosen}
	Let $m \neq p$ be a prime, and let $u \geq 1$ be an integer coprime with $p$ and $m$. Then for each integer $k \geq 1$ and $r \in \llbracket0, m-1\rrbracket$, the curves %defined over $\F_p$ given by
	\begin{equation*}\mathcal{F} : y^m = x^{m^ku} + 1 \mbox{\; and \;} \mathcal{G}_r : y^m = x^r(x^{m^{k-1}u} + 1)\end{equation*}
	are such that
	\begin{equation*}\gamma(\mathcal{F}) = \sum_{r=0}^{m-1} \gamma(\mathcal{G}_{r}).\end{equation*}
	%In particular, if $m=2$ then  $\gamma(\mathcal{G}_{1}) = \gamma(\mathcal{F}_{2,2^ku}) - \gamma(\mathcal{F}_{2,2^{k-1}u})$.
\end{theorem}
\begin{proof}
	Let $\zeta$ be a primitive $m$-th root of unity. For $k \geq 1$, note that 
	\begin{equation*}\pi: (x,y) \mapsto (\zeta x, y) \mbox{\;and\;} \sigma : (x,y) \mapsto (x, \zeta^{m-1}  y)\end{equation*}
	%$$\pi = \Big\{ \begin{array}{l}    x' = \zeta x \\     y' = y \end{array}  \mbox{\;and\;} \sigma =\Big\{ \begin{array}{l}    x' = x \\
	%y' =\zeta^{m-1}  y
	%\end{array} $$
	are automorphisms of $\mathcal{F}$ and that $G = \langle \pi, \sigma \rangle \cong \Za_m \times \Za_m$ is a subgroup of $\aut(\mathcal{F})$, with
	\begin{equation*}
	G = \langle \sigma \rangle \cup \langle \pi \rangle \cup \langle \pi \circ \sigma \rangle \cup \cdots \cup \langle \pi \circ \sigma^{m-1} \rangle.
	\end{equation*}
	Clearly, each pairwise intersection of these $m+1$ subgroups is trivial. Moreover, one can readily check that  the quotient curves $\mathcal{F}/G$ and $\mathcal{F}/\langle \sigma \rangle$ have genus zero. For  $r \in \llbracket0, m-1\rrbracket$, note that $\varepsilon = x^m$ and $\eta = x^ry$ are elements in $\Kk(x,y)$ that are fixed by
	$\pi \circ \sigma^r$,  that is, $\Kk(\varepsilon,\eta)  \subseteq  \Kk(x,y)^{\langle \pi \circ \sigma^r \rangle} $.
	Since $\Kk(x,y) = \Kk(\varepsilon,\eta)(x)$ and as  $T^m-\varepsilon \in \Kk(\varepsilon,\eta)[T]$ vanishes at $x$, we have that 
	\begin{equation*}[\Kk(x,y) : \Kk(\varepsilon,\eta)] = [\Kk(x,y): \Kk(x,y)^{\langle \pi \circ \sigma^r \rangle}] = m,\end{equation*} 
	and then $\Kk(x,y)^{\langle \pi \circ \sigma^r \rangle} = \Kk(\varepsilon,\eta)$. Therefore, the quotient curve
	%Since $\pi \circ \sigma^r(\varepsilon) = \varepsilon$ and $\pi \circ \sigma^r(\eta) = \eta$. Note that $\Kk(x,y) = \Kk(\varepsilon,\eta)(x)$ and $T^m-\varepsilon$ is a polynomial in $\Kk(\varepsilon,\eta)(T)$ for which $x$ is a root. Therefore, $[\Kk(x,y),\Kk(\varepsilon,\eta)] = [\Kk(x,y),\Kk(x,y)^{\langle \pi \circ \sigma^r \rangle}] = m$, and consequently, $\Kk(x,y)^{\langle \pi \circ \sigma^r \rangle} = \Kk(\varepsilon,\eta)$. Hence 
	$\mathcal{F}/\langle \pi \circ \sigma^r \rangle$ is isomorphic to the curve given by  $\eta^m = \varepsilon^r(\varepsilon^{m^{k-1}u}+1)$. Now Theorem \ref{teo-13032020} yields
	%$$J_{\mathcal{F}}^m \sim J_{\mathcal{F}/H_0}^m \times J_{\mathcal{F}/H_1}^m \times \cdots \times J_{\mathcal{F}/H_{m-1}}^m$$
	\begin{equation*}J_{\mathcal{F}}^m \sim J_{\mathcal{F}/\langle \pi \rangle}^m \times J_{\mathcal{F}/\langle \pi \circ \sigma \rangle}^m \times \cdots \times J_{\mathcal{F}/\langle \pi \circ \sigma^{m-1} \rangle}^m,\end{equation*}
	and then $\gamma(\mathcal{F}) = \dsum_{r=0}^{m-1} \gamma(\mathcal{G}_{r}).$
\end{proof}

\begin{corollary}\label{p-rank222a}\label{p-rank222}\label{cor-04022020}
	Let $u \geq 1$ be an odd integer coprime to $p>2$. Then for each integer $k \geq 0$, the curves 
	\begin{equation*}\mathcal{F}_k : y^2 = x^{2^ku} + 1 \mbox{ \;and\; } \mathcal{H}_k : y^2 = x(x^{2^ku} + 1)\end{equation*}
	are such that
	\begin{enumerate}[\rm(i)]
		\item $\gamma(\mathcal{F}_0) = \gamma(\mathcal{H}_0)$
		\item $\gamma(\mathcal{H}_{k-1}) = \gamma(\mathcal{F}_{k}) - \gamma(\mathcal{F}_{k-1})$, for $k \geq 1$
		\item $\gamma(\mathcal{F}_{k}) = \gamma(\mathcal{F}_0) + \displaystyle\sum_{i=0}^{k-1} \gamma(\mathcal{H}_i)$, for $k \geq 1$
		\item  $\gamma(\mathcal{F}_1) = 2\gamma(\mathcal{F}_0)$.
	\end{enumerate}
\end{corollary}
\begin{proof}
	Note that the map $(x,y) \mapsto \big(\frac{1}{x}, \frac{y}{x^{(u+1)/2}}\big)$ from $\mathcal{H}_0$ to $\mathcal{F}_0$ gives $\mathcal{H}_0 \cong \mathcal{F}_0$,  which proves the first assertion.
	By Theorem \ref{p-rank-Fnm-kani-rosen}, we have
	\begin{equation}\label{eq-13032020}
	\gamma(\mathcal{F}_k) =  \gamma(\mathcal{F}_{k-1}) + \gamma(\mathcal{H}_{k-1}),
	\end{equation}
	which proves (ii). %In addition, (i) and (ii) clearly imply (iii). 
	Assertion (iii) follows from the recursive  relation in \eqref{eq-13032020}, and the last assertion follows  from (i) and the case 
	$k=1$ in (ii). 
\end{proof}

\section{Formulas for the p-rank of Fermat-type curves }\label{sec3}

We have seen in Theorem \ref{cor-p-rank-Cnm} that the $p$-rank of $\mathcal{F}_{m,n}$  can be obtained from the cardinality of a set $T$, given in terms of certain binomial coefficients. In this section,  we show how one can effectively turn this information into formulas for the  $p$-rank of $\mathcal{F}_{m,n}$. 

\begin{theorem}\label{p+1,p-1}
If  $m|(p-1)$ and    $n|(p+1)$, then  $\mathcal{F}_{m,n}: y^m=x^n+1$ has $p$-rank  
	\begin{equation}\label{p-rank-divide}
\gamma(\mathcal{F}_{m,n})=2\sum_{j=\left \lceil{  m/2}\right \rceil }^{m}  \left \lfloor{\frac{  jm_0+1   }{n_0}}\right \rfloor - (m- \left \lceil{ m/2}\right \rceil +2)  (n-1)+1-\operatorname{gcd}(m, n),
\end{equation}
where   $m_0=(p-1)/m$ and $n_0=(p+1)/n$. In particular, 
	\begin{equation}\label{case-n} 
	\gamma(\mathcal{F}_{m,p+1}) = \left\{
	\begin{array}{ll}
	 m_0(m - 1)^2/4 +(m - 1)/2  & \mbox{\; if  $m$ is odd \;} \\
	m_0(m^2 - 2m)/4 + (m - 2)/2  & \mbox{\; if  $m$ is even \;} 
	\end{array}
	\right.\end{equation}
	and
	\begin{equation}\label{case-m} 
	 \gamma(\mathcal{F}_{p-1,n})= \left\{
	\begin{array}{ll}
	n_0(n^2-1)/4-n +1  & \mbox{\; if  $n$ is odd \;} \\
	n_0n^2/4-n   & \mbox{\; if  $n$ is even. \;} 
	\end{array}
	\right.\end{equation}

\end{theorem}  
\begin{proof}
For $\alpha=n_0(p-1)$,  $\beta=m_0(p+1)$,   and
\begin{equation*}
	T = \left\{ (i,j) \in \mathbb{N}^2 \mid  0 \leq i\alpha \leq j\beta \leq p^2-1 \; \mbox{and}\; \dbinom{j\beta}{i\alpha} \not\equiv 0 \pmod{p} \right\},
	\end{equation*}
Theorem \ref{cor-p-rank-Cnm} gives
\begin{equation}\label{p-posto}
\gamma(\mathcal{F}_{m,n})=\# T-(m+n+\operatorname{gcd}(m, n)).
\end{equation}
Note that the values  $i=0$ and $i=n$ give rise to a total of $m+2$ pairs $(i,\,j)\in T$.  For  $1\leq i \leq n-1$,
the base-$p$ expansions  of $i\alpha$ and  $j\beta$  are given by
$$ i\alpha=(in_0-1)p+p-in_0 \quad  \text{  and  }   \quad     j\beta=jm_0p+jm_0,$$
and Lucas' Lemma entails $(i,j) \in T$ if and only if  $p- jm_0\leq in_0\leq jm_0 +1$, that is,
$$n-\left \lfloor{ \frac{ jm_0+1 }{n_0} }\right \rfloor \leq i\leq   \left \lfloor{\frac{  jm_0+1   }{n_0}}\right \rfloor.$$
Hence $\#T=m+2+ \sum_{j=\left \lceil{  m/2}\right \rceil }^{m}\Big(2\left \lfloor{\frac{  jm_0+1   }{n_0}}\right \rfloor-n+1\Big)$,
and then  equation  \eqref{p-posto} concludes the proof of  \eqref{p-rank-divide}. The particular cases corresponding to the curves
$\mathcal{F}_{m,p+1}: y^m=x^{p+1}+1$  and  $\mathcal{F}_{p-1,n}: y^{p-1}=x^{n}+1$ follow directly from the  computation of
$\sum_{j=\left \lceil{  m/2}\right \rceil }^{m}  \left \lfloor{\frac{  jm_0+1   }{n_0}}\right \rfloor $. Case \eqref{case-n} follows immediately from
$$\sum_{j=\left \lceil{  m/2}\right \rceil }^{m} (  jm_0+1 )= (m-\left \lceil{  m/2}\right \rceil +1)(m_0m+m_0\lceil{  m/2}\rceil +2)/2.$$
Case  \eqref{case-m} follows from the basic properties of the floor function, including the identity  $\sum_{i=0}^{rs-1} \Big\lfloor  \frac{i  }{s}\Big\rfloor=sr(r-1)/2$,
from which one obtains 
\begin{equation*} \sum_{j=  \left \lceil{  (p-1)/2}\right \rceil}^{p-1}  \left \lfloor{\frac{  j+1   }{n_0}}\right \rfloor = \left\{
	\begin{array}{ll}
	 n_0(3n^2-2n-1)/8  & \mbox{\; if  $n$ is odd \;} \\
	n_0(3n^2-2n)/8   & \mbox{\; if  $n$ is even. \;} 
	\end{array}
	\right.\end{equation*}
\end{proof}

\begin{theorem}\label{c-m,p^h-1}
	Let $m, h \in \Za_{>0}$ be such that $m >1$ divides $p^h-1$. The $p$-rank of
	$\mathcal{F}_{m,p^h-1}: y^m = x^{p^h-1} + 1$
	is
	\begin{equation}\label{c-m,p^h-1 eq 1}
	\gamma = \sum_{j=1}^{m-1} \Big( \prod_{r=1}^h (a_{j,r}+1) \Big) - 2(m-1),
	\end{equation}
	where $a_{j,1}, \ldots, a_{j,h}$ are the coefficients in the base-$p$ expansion of $j(p^h-1)/m$.
	%where $j(p^h-1)/m = a_{j,1}p^{h-1} + \cdots + a_{j,h-1}p + a_{j,h}$, and $a_{j,i} \in \{0, \ldots, p-1\}$, for all $1 \leq j \leq m-1$ and $1 \leq r \leq h$.
	In particular, if $m$ divides $p-1$, then
	\begin{equation}\label{c-m,p^h-1 eq 2}
	\gamma = \sum_{j=1}^{m-1} \Big( \dfrac{j(p-1)}{m} + 1 \Big)^h - 2(m-1).
	\end{equation}
\end{theorem}

\begin{proof}
	From Theorem \ref{cor-p-rank-Cnm}, we have  $\beta = (p^h-1)/m$ and $\gamma = \#T- (2m+p^h-1)$, where 
	\begin{equation*}\label{c-m,p^h-1 eq 3}
	T = \left\{ (i,j) \in \Na^2 \mid  0 \leq i \leq j \beta \leq p^h-1
	\mbox{\;and\;} \dbinom{j \beta}{i} \not\equiv 0 \pmod{p} \right\}.
	\end{equation*}
	% and 
	%	$$T_j = \left\{ (i,j) \in \Na^2 : 0 \leq i \leq j \beta   \mbox{\;and\;} \dbinom{j \beta}{i} \not\equiv 0 \pmod{p} \right\},\; \mbox{for all}\; 0 \leq j \leq m.$$
	For each $j \in \llbracket0, m\rrbracket$, consider the base-$p$ expansion of $j \beta =  a_{j,1}p^{h-1} + \cdots + a_{j,h}$.
	%Let us fix $j \beta =  a_{j,1}p^{h-1} + \cdots + a_{j,h-1}p + a_{j,h}$, with $a_{j,r} \in \{0, \ldots, p-1\}$ for all $0 \leq j \leq m$ and $0 \leq r \leq h$.
	By Lucas' Lemma, the integers $i \in \llbracket0, j \beta\rrbracket$ for which $\binom{j\beta}{i} \not\equiv 0 \pmod{p}$ are given by
	$i = b_1p^{h-1} + \cdots + b_{h-1}p + b_h$,
	where  $b_r \in \llbracket0, a_{j,r}\rrbracket$ and  $1 \leq r \leq h$.  
	Clearly, there exist $\prod_{r=1}^h (a_{j,r}+1)$ such integers. Therefore,
	\begin{equation*}\gamma = \sum_{j=0}^{m} \Big( \prod_{r=1}^h (a_{j,r}+1) \Big) - (2m+p^h-1) = \sum_{j=1}^{m-1} \Big( \prod_{r=1}^h (a_{j,r}+1) \Big) - 2(m-1).\end{equation*}
	Note that if $m \mid (p-1)$, then $a_{j,r} = j(p-1)/m$ for all $r \in \llbracket1, h\rrbracket$. Thus  (\ref{c-m,p^h-1 eq 2}) follows immediately from  (\ref{c-m,p^h-1 eq 1}).
\end{proof}
\begin{theorem}\label{h=3-26072022}
	The $p$-rank of the curves
	\begin{equation*}
	y^{p^2+p+1} = x^{p^2+p+1} + 1 \; \mbox{  and  }\; y^{p^2+p+1} = x^{p^{3}-1} + 1 
	\end{equation*}
	are given by
	\begin{equation*}
	\gamma(\mathcal{F}_{p^2+p+1})= p(p+1)(p^2 + p + 2)/8
	\; \mbox{  and  }\;
	\gamma(\mathcal{F}_{p^2+p+1,p^{3}-1})=p(p+1)(p^3 + 2p^2 + 3p - 14)/8,
	\end{equation*}
	respectively.
\end{theorem}
\begin{proof}
	It follows from \eqref{eq1-13042021} 
	that $\gamma(\mathcal{F}_{p^2+p+1})=\#A$, where
	\begin{equation*}
	A = \left\{ (i,j) \in \Na^2 \mid  1 \leq i < j \leq p^2+p \; \mbox{and}\; \dbinom{j(p-1)}{i(p-1)} \not\equiv 0 \pmod{p} \right\}.
	\end{equation*}
	For each $i,j \in \llbracket1, p^2+p\rrbracket$, write $j = a_1p+a_0$ and $i = b_1p+b_0$, with $a_r,b_r \in \llbracket0, p\rrbracket$ and $a_0b_0 \neq 0$,  and consider the following base-$p$ expansions 
	\begin{equation}\label{18112022-1}
	j(p-1) = \left\{ \begin{array}{ll}
	(a_1-1)p^2+(p+a_0-a_1-1)p+(p-a_0) & \mbox{if}\; a_0 \leq a_1\\
	a_1p^2+(a_0-a_1-1)p+(p-a_0) & \mbox{if}\; a_0 > a_1
	\end{array}
	\right.
	\end{equation}
	and 
	\begin{equation*}
	i(p-1) = \left\{ \begin{array}{ll}
	(b_1-1)p^2+(p+b_0-b_1-1)p+(p-b_0) & \mbox{if}\; b_0 \leq b_1\\
	b_1p^2+(b_0-b_1-1)p+(p-b_0) & \mbox{if}\; b_0 > b_1.
	\end{array}
	\right.
	\end{equation*}
	By Lucas' Lemma, there  is  no  pair $(i,j)=(b_1p+b_0, a_1p+a_0) \in A$ for which $a_0 > a_1$ or $b_0 \leq b_1$.  Indeed, direct inspection shows that $\binom{j(p-1)}{i(p-1)} \not\equiv 0 \pmod{p}$ implies that $i=j$ when either $a_0 > a_1$ and $b_0 > b_1$ or $a_0 \leq a_1$ and $b_0 \leq b_1$, and that $\binom{j(p-1)}{i(p-1)} \equiv 0 \pmod{p}$ whenever $a_0 > a_1$ and $b_0 \leq b_1$.
	Thus if  $(i,j) = (b_1p+b_0,a_1p+a_0) \in A$,  then $1 \leq a_0 \leq a_1 \leq p$  and    $0\leq b_1<b_0\leq p$,  and  Lucas' Lemma yields
	
	\begin{equation*}
	A = \left\{\left(b_1p+b_0,a_1p+a_0\right)  \ \middle\vert  
	\begin{array}{l}
	a_1,b_0 \in \llbracket 1, p\rrbracket,\\
	1 \leq a_0, b_1+1 \leq \min\left\{a_1, b_0\right\}, \\ 
	a_1+b_0 \leq p+a_0+b_1 \\
	\end{array}
	\right\}.\end{equation*}
	Therefore, by Lemma \ref{20112022-1},  $\gamma(\mathcal{F}_{p^2+p+1})= \#A = p(p + 1)(p^2 + p + 2)/8$.

	For the second  curve,  note that Theorem \ref{c-m,p^h-1} gives
	\begin{equation}\label{18112022-2}
	\gamma(\mathcal{F}_{p^2+p+1,p^{3}-1}) = \sum_{j=1}^{p^2+p} (a_{j,1}+1)(a_{j,2}+1)(a_{j,3}+1) - 2(p^2+p),
	\end{equation}
	%\textcolor{red}{ produt\'orio pra 3 elementos fica meio estranho}
	where $a_{j,1}, a_{j,2},$ and  $a_{j,3}$ are the coefficients in the base-$p$ expansion of $j(p-1)$. Since $j \in \llbracket1, p^2+p\rrbracket$, such coefficients are given by \eqref{18112022-1}, %it follows that 
	where $a_0,a_1 \in \llbracket0, p\rrbracket$ and $a_0 \neq 0$. Therefore, %$a_{j,3} = p-a_0$,
	%\begin{equation*}
	%a_{j,2} = \left\{ \begin{array}{ll}
	%p+a_0-a_1-1 & \mbox{if}\; a_0 \leq a_1\\
	%a_0-a_1-1 & \mbox{if}\; a_0 > a_1
	%\end{array}
	%\right.\; \mbox{and} \;\;
	%a_{j,1} = \left\{ \begin{array}{ll}
	%a_1-1 & \mbox{if}\; a_0 \leq a_1\\
	%a_1& \mbox{if}\; a_0 > a_1
	%\end{array}
	%\right.,
	%\end{equation*}
	\begin{eqnarray*}   
		\gamma(\mathcal{F}_{p^2+p+1,p^{3}-1})+2(p^2+p) & = & \sum_{a_1=0}^p \Big( \sum_{a_0=1}^{a_1} a_1(p+a_0-a_1)(p-a_0+1) + \sum_{a_0=a_1+1}^{p} (a_1+1)(a_0-a_1)(p-a_0+1)\Big) \\
		& = & p(p + 1)^{2}(p^2 + p + 2)/8,
	\end{eqnarray*}
	%which leads us to}	
	%\begin{equation*}
	%\displaystyle\sum_{j=1}^{p^2+p} (a_{j,1}+1)(a_{j,2}+1)(a_{j,3}+1) = p(p + 1)^{2}(p^2 + p + 2)/8.
	%\end{equation*}
	%	\textcolor{red}{ essa igualdade a\'i acima  precisa explicar mais. Comparar com o que est\'a feito em outros pontos do paper}\; \textcolor{blue}{Pensei em apontar os elementos relevantes na conta.}
	and then $\gamma(\mathcal{F}_{p^2+p+1,p^{3}-1})=p(p+1)(p^3 + 2p^2 + 3p - 14)/8$. 
\end{proof}

%\begin{corollary}\label{h=3-26072022}
%	The p-rank of $y^{p^2+p+1} = x^{p^{3}-1}+1$  is 
%	\begin{equation*}\gamma=p(p+1)(p^3 + 2p^2 + 3p - 14)/8.\end{equation*}
%\end{corollary}
%\begin{proof}
%	It follows from  Theorem \eqref{c-m,p^h-1} that
%	\begin{equation}\label{parcial1}
%	\gamma = \sum_{j=1}^{p^2+p} \Big( \prod_{r=1}^3 (a_{j,r}+1) \Big) - 2(p^2+p),
%	\end{equation}
%	where $a_{j,1}, a_{j,2}, a_{j,3}$ are the coefficients in the base-$p$ expansion of $j(p-1)$. For each $j \in \llbracket1, p^2+p\rrbracket$, write $j = a_{1}p + a_{0}$, with $a_0, a_1 \in \llbracket0, p\rrbracket$ and  $a_0 \neq 0$. Thus the base-$p$ expansion of $j(p-1)$ is given by
%	\begin{equation*}j(p-1) = \left\{ \begin{array}{ll}
%	(a_1-1)p^2+(p+a_0-a_1-1)p+(p-a_0) & \mbox{if}\; a_0 \leq a_1\\
%	a_1p^2+(a_0-a_1-1)p+(p-a_0) & \mbox{if}\; a_1+1 \leq a_0,
%	\end{array}
%	\right.\end{equation*}
%	and then
%	$\displaystyle\sum_{j=1}^{p^2+p} \Big( \displaystyle\prod_{r=1}^3 (a_{j,r}+1) \Big) = p(p + 1)^{2}(p^2 + p + 2)/8.$
%	This together with \eqref{parcial1}  gives the result.
%	%$$=\sum\limits_{a_1=0}^p \Big( \sum_{a_0=1}^{a_1} a_1(p+a_0-a_1)(p-a_0+1) + \sum\limits_{a_0=a_1+1}^{p} (a_1+1)(a_0-a_1)(p-a_0+1)\Big)$$ 
%	%$$=p(p + 1)^{2}(p^2 + p + 2)/8.$$
%\end{proof}

\begin{theorem}\label{cor-12032020a}
	If $h \in \Za_{>0}$ and $p \neq 2$, then the $p$-rank of the curves
	\begin{equation*}
	y^{\frac{p^h-1}{2}} = x^{\frac{p^h-1}{2}} + 1 \; \mbox{  and  }\;
	y^{\frac{p^h-1}{2}} = x^{p^h-1} + 1\end{equation*}
	are given by
	\begin{equation*}\gamma(\mathcal{F}_{(p^h-1)/2})=(p+1)^h(p^h+3)/2^{h+2}-3(p^h-1)/2\end{equation*}
	and 
	\begin{equation*}\gamma(\mathcal{F}_{(p^h-1)/2,p^h-1})= (p+1)^h(p^h+1)/2^{h+1} - 2(p^h-1),\end{equation*}
	respectively.
\end{theorem}
\begin{proof}
	For each integer $h\geq 1$, consider the set
	\begin{equation}\label{T_h}
	T_h= \left\{ (i,j) \in \Na^2 \mid  0 \leq i \leq j \leq p^h-1 \; \mbox{and}\; \dbinom{j}{i} \not\equiv 0 \pmod{p} \right\}.
	\end{equation}
	From Theorem \ref{cor-p-rank-Cnm}, to obtain  $\gamma(\mathcal{F}_{(p^h-1)/2})=(p+1)^h(p^h+3)/2^{h+2}-3(p^h-1)/2$,
	it suffices to prove that the number of elements
	$(i,j) \in T_h$ for which  both $i$ and $j$ are even is  $ (p+1)^{h}   ( p^{h}+   3 )/2^{h+2}$. To this end,  first note that $\#T_1=p(p+1)/2$, and that for  $(i,j) \in T_h$, the base-$p$ expansions  $ j =\sum_{r=0}^{h-1} a_rp^r$ and $ i =\sum_{r=0}^{h-1} b_rp^r$ along with  Lucas' Lemma entail
	\begin{equation}\label{CaseF}
	\#T_h=  (\#T_1)^h=\left(p(p+1)/2\right)^h.
	\end{equation}
	Now for $u,v\in \{0,1\}$,  let us consider the partition of $T_h$ given by the  four sets 
	\begin{equation*}
	T_h^{(u,v)} = \left\{ (i,j) \in T_h \,\mid (i,j)\equiv (u,v) \pmod{2} \right\},
	\end{equation*}
	where $(i,j)\equiv (u,v) \pmod{2}$ stands for  $i \equiv u \pmod{2}$ and   $j \equiv v \pmod{2}$.	
	We need to prove that  $\#T_h^{(0,0)}= (p+1)^{h}  (p^h+   3)/2^{h+2}$, and for this we proceed by induction on $h$.  The case $h=1$ is given  by the following easily checked fact.
	\begin{equation}\label{caseT1}
	\#T_1^{(u,v)}=\left\{
	\begin{array}{ll}
	(p+1)(p+3)/8 & \mbox{if  $(u,v)=(0,0)$}   \vspace{0.2cm}\\
	(p^2-1)/8 & \mbox{otherwise}.
	\end{array}
	\right.
	\end{equation}
	Now suppose $h\geq 2$. For $ j =\sum_{r=0}^{h-1} a_rp^r$ and $ i =\sum_{s=0}^{h-1} b_sp^s$, since  $(i,j)\equiv ( \sum_{r=0}^{h-1} a_r,\sum_{r=0}^{h-1} b_r)  \pmod{2}$,
	Lucas' Lemma implies that  $(i,j) \in T_h^{(0,0)}$ if and only if 
	\begin{equation*}(a_{h-1},b_{h-1}) \in T_1^{(u,v)}\quad  \text{ and } \quad  \Big(\sum\limits_{s=0}^{h-2} b_sp^s,\sum\limits_{r=0}^{h-2} a_rp^r\Big)  \in T_{h-1}^{(u,v)}\end{equation*}
	for some $(u,v) \in \llbracket0,1\rrbracket^2 $. Since $T_{h-1}$ is the disjoint union of  $T_{h-1}^{(0,0)}$ and $\bigsqcup\limits_{(u,v)\neq (0,0)} T_{h-1}^{(u,v)}$, equation \eqref{caseT1} provides
	\begin{eqnarray*}
		\#T_h^{(0,0)} &=&  \big( (p+1)(p+3) \#T_{h-1}^{(0,0)}+(p^2-1)(\#T_{h-1}-\#T_{h-1}^{(0,0)})\big)/8\\
		\	                    &=& (p+1)  \big(  4 \#T_{h-1}^{(0,0)}+(p-1)\#T_{h-1} \big)/8\\
		&=&  (p+1)^{h}   ( p^{h}+   3 )/2^{h+2},
	\end{eqnarray*}
	where the latter equality follows from   $ \#T_{h-1}^{(0,0)}=(p+1)^{h-1}(p^{h-1}+   3)/2^{h+1}$, given by the induction hypothesis, and from 
	$\#T_{h-1}=(p(p+1)/2)^{h-1}$ in \eqref{CaseF}.
	
	For the curve  $\mathcal{F}_{p^h-1,(p^h-1)/2}$, in accordance  with Theorem \ref{cor-p-rank-Cnm}, it only remains to prove that the total number of 
	$(i,j) \in T_h$ for which  $j$ is even, that is,  $\#T_{h}^{(0,0)}+\#T_{h}^{(0,1)}$,   is    $(p+1)^{h}(p^{h}+1)/2^{h+1}$. To this end,
	note that the maps
	\begin{equation*} T_{h}^{(1,0)} \longrightarrow T_{h}^{(0,1)}, \hspace{0.1cm} (i,j) \mapsto (p^h-1-j,p^h-1-i) \quad \text{and} \quad  T_{h}^{(0,1)} \longrightarrow T_{h}^{(1,1)}, \hspace{0.1cm} (i,j) \mapsto (j-i,j)\end{equation*} 
	are bijections, given that   $p^h-1=\sum_{r=0}^{p-1}(p-1)p^r$ is even and that $\binom{j}{i}=\binom{j}{j-i}$.  Therefore, since  $\#T_{h}=p^h(p+1)^h/2^h$  and   $\#T_{h}^{(0,0)}= (p+1)^{h} (  p^{h}+   3)/2^{h+2}$, we have 
	\begin{equation*}\#T_{h}^{(1,0)}=\#T_{h}^{(0,1)}=\#T_{h}^{(1,1)}=(p+1)^{h}(p^h-1)/2^{h+2}.\end{equation*} 
	Hence 
	\begin{equation*}\#T_{h}^{(0,0)}+\#T_{h}^{(0,1)}= (p+1)^{h}(p^{h}+1)/2^{h+1},\end{equation*}
	and the result follows.
\end{proof}

From Corollary \ref{p-rank-Fn-Fermat}, the  $p$-rank of the Fermat curve $\mathcal{F}_{p^h-1}:\,y^{p^h-1}=x^{p^h-1}+1$ is $\#T_h-3(p^h-1)$,
where $T_h$ is the set given in \eqref{T_h}. Thus equation  \eqref{CaseF}  implies  that the  $p$-rank of $\mathcal{F}_{p^h-1}$ is $(p(p+1)/2)^h-3(p^h-1)$, a result   proved by Bassa and Beelen in \cite[Theorem 19]{Bassa}. In what follows, we provide a preliminary step to obtain formulas for the  $p$-rank of the curves 
\begin{equation*}y^{p^m\pm 1} = x^{p^{n}\pm 1} + 1,\end{equation*}  natural generalizations of
the Fermat curve $y^{p^h-1}=x^{p^h-1}+1$.

\begin{lemma}\label{13102022-1}
	Let $u, v \in \Za_{>0}$ and  $w=\gcd(u,v)$. Suppose that  $I, J  \in  \llbracket0, p^{vu/w}-1\rrbracket$ have  base-$p$ expansions given by
	\begin{equation}\label{types}
	I = \displaystyle\sum_{b=0}^{\frac{v}{w}-1} \displaystyle\sum_{s=0}^{u-1} b_s p^{ub+s}     \;\mbox{and}\;    J = \displaystyle\sum_{a=0}^{\frac{u}{w}-1} \displaystyle\sum_{r=0}^{v-1} a_rp^{va+r}.
	\end{equation}
	Then the following are equivalent.
	\begin{enumerate}[\rm (i)]
		\item $\dbinom{J}{I} \not\equiv 0 \pmod{p}$
		\item For every $(s,t) \in  \llbracket0, u-1\rrbracket \times \llbracket0, w-1\rrbracket$ such that $s \equiv t \pmod{w}$, we have
		\begin{equation*}0 \leq b_s \leq \mymin \left\{a_{\theta w+t} \mid \theta \in \llbracket0, u/w-1\rrbracket \right\}.\end{equation*} 
	\end{enumerate}		
%	Moreover, for any  such  a pair $I,J$ with   $\dbinom{J}{I} \not\equiv 0 \pmod{p}$, there corresponds a unique 
	%pair $ \tilde{I}, \tilde{J}$ given by \eqref{types}  such that  $\dbinom{\tilde{I}}{\tilde{J}} \not\equiv 0 \pmod{p}$.
\end{lemma}
\begin{proof}
	Fix $s \in \llbracket0, u-1\rrbracket$.  In  the base-$p$  expansion of  $I$, the  powers $p^{ub+s}$, with $b \in \llbracket0, v/w-1\rrbracket$, have the  same coefficient $b_s$.
	The corresponding  coefficients of such powers  in the base-$p$  expansion of  $J$ are the   $v/w$  coefficients   $a_{r}$, with $r  \in \llbracket0, v-1\rrbracket$, such that    $ r  \equiv  ub+s   \pmod{v}$.
	One can check that the set of such values  of  $r$ is given by 
	\begin{equation}\label{possib}
	\{r_1,\dots,r_{v/w}\}:=\{\theta w+t  \mid  \theta \in \llbracket0, v/w-1\rrbracket  \},
	\end{equation}
	where $t \equiv  s \pmod{w}$   and $t  \in \llbracket0, w-1\rrbracket$.
	Note that this set depends  only on the reduction of $s \in \llbracket0, u-1\rrbracket$  modulo $w$.  Therefore,  for any fixed  $t\in \llbracket0, w-1\rrbracket$, there are  exactly 
	$u/w$ values   $s_1, \ldots,s_{u/w} \in \llbracket0, u-1\rrbracket$, with  $s_i \equiv t \pmod{w}$, that give rise to the same
	set  \eqref{possib}. That is,  for all $b \in \llbracket0, v/w-1\rrbracket$,  the set of coefficients of $p^{ub+s_1},\ldots, p^{ub+s_{n/d}} $ in the base-$p$ expansion of  $J$  is given by  $\{a_{r_1},\dots,a_{r_{v/w}}\}$. Thus by Lucas' Lemma, $\binom{J}{I} \not\equiv 0 \pmod{p}$ if and only if for every $s \in  \llbracket0, u-1\rrbracket$, we have $0 \leq b_s \leq \mymin \left\{a_{\theta w+t} \mid \theta \in \llbracket0, u/w-1\rrbracket \right\}$, where $t \equiv s \pmod{w}$ and $t  \in \llbracket0, w-1\rrbracket$. 
	
	%For the final assertion,  note that if  $I$ and $J$ have  base-$p$ expansions given by \eqref{types}, then
	%\begin{equation*}  \tilde{I}:=p^{uv/w}-1-I       \text{   and   }    \tilde{J}:=p^{uv/w}-1-J  \end{equation*}   have that same type of   base-$p$ expansion, and 
	%\begin{equation*}
	%\dbinom{J}{I} \not\equiv 0 \pmod{p}   \iff  \dbinom{\tilde{I}}{\tilde{J}} \not\equiv 0 \pmod{p}.
	%\end{equation*}
\end{proof}

\begin{theorem}\label{+formulas-thm1}
	Let $n, m \in \Za_{>0}$, and $d=\gcd(n,m)$. Then the $p$-rank of
	$\mathcal{F}: y^{p^m-1} = x^{p^n-1} + 1$ is
	\begin{equation*}\gamma = \Big(\dsum_{i=0}^{p-1} \big((i+1)^{\frac{m}{d}}-i^{\frac{m}{d}}\big)(p-i)^{\frac{n}{d}}\Big)^d - (p^m+p^n+p^d-3).\end{equation*}
\end{theorem}

\begin{proof}
	Let $q=p^{mn/d}$, $\alpha = (q-1)/(p^n-1)$, and $\beta = (q-1)/(p^m-1)$.  	By Theorem \ref{cor-p-rank-Cnm}, the $p$-rank of $\mathcal{F}$ is $\gamma=\#T-(p^m+p^n+p^d-3)$, where
	\begin{equation*}
	T = \left\{ (i,j) \in \Na^2 \mid  0 \leq i\alpha \leq j\beta \leq q-1 \; \mbox{and}\; \dbinom{j\beta}{i\alpha} \not\equiv 0 \pmod{p} \right\}.
	\end{equation*}
	For $(i,j) \in T$, consider the base-$p$ expansions
	\begin{equation*}j = a_{m-1}p^{m-1} + \cdots + a_{1}p + a_{0} \; \mbox{and}\; i = b_{n-1}p^{n-1} + \cdots + b_{1}p + b_{0},\end{equation*} and note that 
	\begin{equation*}j\beta = \displaystyle\sum_{a=0}^{\frac{n}{d}-1} \displaystyle\sum_{r=0}^{m-1} a_rp^{ma+r}\; \mbox{and}\; i\alpha = \displaystyle\sum_{b=0}^{\frac{m}{d}-1} \displaystyle\sum_{s=0}^{n-1} b_s p^{nb+s}.\end{equation*}
	Thus by Lemma \ref{13102022-1}, we have $\#T = \#L$, where 
	\begin{equation*}L=\left\{\left(a_0, \ldots, a_{m-1}, b_0, \ldots, b_{n-1}\right)  \ \middle\vert  
	\begin{array}{l}
	a_r \in \llbracket0, p-1\rrbracket,\\
	0 \leq b_s \leq \min\left\{a_{\theta d+t}\mid 
	\theta \in \llbracket0, m/d-1\rrbracket 
	\right\}, \\ 
	t \in \llbracket0, d-1\rrbracket, \mbox{\;and\;} t \equiv s \pmod{d} \\
	\end{array}
	\right\}.\end{equation*}
	Let $t\in \llbracket0, d-1\rrbracket$. Note that for each $u = (a_0, \ldots, a_{m-1}, b_0, \ldots, b_{n-1}) \in A$, the set of the indices $\{{r_1}, \ldots, {r_{m/d}}, {s_1}, \ldots, {s_{n/d}}\}$ (cf. proof of Lemma \ref{13102022-1}), whose corresponding entries in $u$ are such that $0 \leq b_{s_j} \leq \min  \left\{a_{r_i}\mid \, i \in \llbracket1, m/d\rrbracket\right\}$, is characterized by the condition $r_i \equiv s_j \equiv t \pmod{d}$ for all $(i,j)  \in \llbracket1, m/d\rrbracket \times  \llbracket1, n/d\rrbracket$. Thus rearranging the entries of $u$, the set $L$ can be seen as the cartesian product of  the $d$ sets
	\begin{equation*}L_t = \left\{\left(a_{r_1}, \ldots, a_{r_{m/d}}, b_{s_1}, \ldots, b_{s_{n/d}}\right)\ \middle\vert     \begin{array}{l}  a_{r_i} \in \llbracket0, p-1\rrbracket, 
	\\ 0 \leq b_{s_j} \leq \min  \left\{a_{r_i}\mid \, i \in \llbracket1, m/d\rrbracket\right\},
	\\ r_i \equiv s_j \equiv t \pmod{d}\end{array}\right\},\end{equation*}
	each of which has cardinality
	\begin{equation*}\#L_t=\sum\limits_{i=0}^{p-1}\left((i+1)^{m/d}-i^{m/d}\right)(p-i)^{n/d},\end{equation*}
	given by Lemma \ref{+formulas-lema1}. Therefore,  $\#T=\#L=(\#L_t)^d$, and the result follows.
\end{proof}

\begin{corollary}\label{+formulas-thm1-cor1}
	Let $n, m \in \Za_{>0}$ be odd, and let $u$ and $v$ be the orders of $2$ in $(\Za/m\Za)^{\times}$ and $(\Za/n\Za)^{\times}$, respectively. If $\gcd(u,v)=1$, then the $2$-rank of
	$\mathcal{F}: y^{m} = x^{n} + 1$ is zero.
\end{corollary}

\begin{proof}
	Considering the morphism 
	\begin{equation*}\mathcal{F}_{2^{u}-1,2^{v}-1}\longrightarrow \mathcal{F}, \quad (x,y) \mapsto (x^{(2^{v}-1)/n},y^{(2^{u}-1)/m}),\end{equation*}
	it follows from item (ii) of Lemma \ref{p-rank-morphism} that we only need to prove that $\gamma(\mathcal{F}_{2^{u}-1,2^{v}-1})=0$, which  is given by  Theorem \ref{+formulas-thm1}, since $\gcd(u,v)=1$. 
\end{proof}

Let $n, m \in \Za_{>0}$. As  mentioned in Remark \ref{supersing}, the  supersingularity of the curves $\mathcal{F}: y^{p^m+1} = x^{p^n+1} + 1$  depends on whether or not $\nu_2(n) =\nu_2(m)$, where  $\nu_2: \mathbb{Z} \rightarrow \mathbb{N} \cup\{\infty\}$ is the $2$-adic valuation.  Nevertheless, in the following  theorem, we show that, in any case, such curves  have zero $p$-rank.

\begin{theorem}\label{+formulas-thm4}
	Let $n, m \in \Za_{>0}$ be divisors of $p^{u}+1$ and $p^v+1$, respectively, where $u, v \in \Za_{\geq 0}$. Then $\mathcal{F}: y^{m} = x^{n} + 1$ has zero $p$-rank.
\end{theorem}
\begin{proof}
	Considering the map $\mathcal{F}_{p^v+1,p^u+1}\longrightarrow \mathcal{F}$, $(x,y) \mapsto (x^{(p^u+1)/n},y^{(p^v+1)/m})$, and item  (ii) of Lemma \ref{p-rank-morphism}, it suffices to prove that $\gamma(\mathcal{F}_{p^v+1,p^u+1})=0$.
	From \eqref{eq1-13042021}, we have  $\gamma(\mathcal{F}_{p^v+1,p^u+1})=\#A$, where %$A$ is given by
	\begin{equation*} A = \left\{ (i,j) \in \Na^2 \mid  p^v+1 \leq i(p^v+1) < j(p^u+1) \leq (p^u+1)p^v \; \mbox{and}\; \dbinom{j\beta}{i\alpha} \not\equiv 0 \pmod{p} \right\},\end{equation*}
	with $d=\gcd(v,u)$, $q=p^{2vu/d}$, $\alpha = (q-1)/(p^u+1)$, and $\beta = (q-1)/(p^v+1)$. Suppose $A\neq \emptyset$, and let $(i,j) \in A$. Since $j \in \llbracket1, p^v\rrbracket$ and $i \in \llbracket1, p^u\rrbracket$, we may consider the following base-$p$ expansions
	\begin{equation}\label{base-p-1}
	j-1 = a_{v-1}p^{v-1} + \cdots + a_{1}p + a_{0}\; \mbox{and}\; i-1 = b_{u-1}p^{u-1} + \cdots + b_{1}p + b_{0}.
	\end{equation}
	Observe that
	\begin{equation*}j\beta = \displaystyle\sum_{a=0}^{\frac{u}{d}-1}\displaystyle\sum_{r=0}^{2v-1} c_rp^{2va+r} \;\mbox{and}\; i\alpha = \displaystyle\sum_{b=0}^{\frac{v}{d}-1}\displaystyle\sum_{s=0}^{2u-1} e_s p^{2ub+s},\end{equation*}
	where
	\begin{equation*}c_r = \left\{
	\begin{array}{ll}
	p-1-a_r & \mbox{if\;}  r \in \llbracket 0, v-1\rrbracket\\
	a_{r-v} & \mbox{if\;}  r \in \llbracket v, 2v-1\rrbracket
	\end{array}
	\right. \mbox{\;and\;}
	e_s = \left\{
	\begin{array}{ll}
	p-1-b_s & \mbox{if\;}  s \in \llbracket 0, u-1\rrbracket\\
	b_{s-u} & \mbox{if\;}  s \in \llbracket u, 2u-1\rrbracket.
	\end{array}
	\right.
	\end{equation*}
	Let $(s,w) \in \llbracket0, u-1\rrbracket \times \llbracket0, d-1\rrbracket$, with  $s \equiv w \pmod{d}$. By Lemma \ref{13102022-1}, $\binom{j\beta}{i\alpha} \not\equiv 0 \pmod{p}$ implies
	\begin{eqnarray*}
		0 \leq e_s, e_{u+s} & \leq & \displaystyle\min\{c_{\theta d+w}\mid \theta \in \llbracket0, 2v/d-1\rrbracket\}\\ &=& \displaystyle\min\{a_{\theta d+w}, p-1-a_{\theta d+w} \mid  \theta \in \llbracket0, v/d-1\rrbracket\}.
	\end{eqnarray*}
	In particular, 
	\begin{equation}\label{meet}
	\max\{b_s,p-1-b_s\} \leq \min\{a_{\theta d+w}, p-1-a_{\theta d+w}  \mid \theta \in \llbracket0, v/d-1\rrbracket\}.
	\end{equation}
	It is easy to see that \eqref{meet} implies $p>2$  and  $b_s=a_{\theta d+w} = (p-1)/2$. Hence from \eqref{base-p-1}, we have $i=(p^u+1)/2$ and    $j=(p^v+1)/2$, which contradicts the condition $i(p^v+1) < j(p^u+1)$ in the definition of the set $A$.
	Therefore, $\gamma(\mathcal{F}_{p^v+1,p^u+1}) = \#A = 0$.
\end{proof}

\begin{theorem}\label{+formulas-thm2}\label{+formulas-thm3}
	Let $m, n \in \Za_{>0}$, $d=\gcd(n,m)$, and  $\gamma$ be the $p$-rank of 
	\begin{equation*}\mathcal{F}: y^{p^m+1} = x^{p^n-1} + 1.\end{equation*}
	\begin{enumerate}[\rm (i)]
		\item If $n/d$ is even, then 
		%{\footnotesize
		\begin{equation*}\gamma =\Big(\dsum_{i=0}^{p-1} \big((p-i) (i+1)\big)^{\frac{n}{2d}} + \dsum_{1 \leq j \leq i \leq p-1} \big((j+1)^{\frac{m}{d}}-2j^{\frac{m}{d}}+(j-1)^{\frac{m}{d}}\big)\big((p-i) (i-j+1)\big)^{\frac{n}{2d}}  \Big)^d - (p^{m}+p^d).\end{equation*}%}
		%$$\gamma =\Big(\dsum_{i=0}^{p-1} (p-i)^{\frac{n}{2d}} \Big( (i+1)^{\frac{n}{2d}} + \dsum_{j=1}^{i} \big((j+1)^{\frac{m}{d}}-2j^{\frac{m}{d}}+(j-1)^{\frac{m}{d}}\big)(i-j+1)^{\frac{n}{2d}} \Big) \Big)^d - (p^{m}+p^d).$$
		\item If $n/d$ is odd, then 
		\begin{equation*}\gamma = \dfrac{1}{2^n}\Big((p+1)^{\frac{n}{d}} + \dsum_{i=1}^{\frac{p-1}{2}} \big((2i+1)^{\frac{m}{d}}-(2i-1)^{\frac{m}{d}}\big)(p-2i+1)^{\frac{n}{d}}\Big)^d - (p^m+1),\end{equation*}
		when $p\neq 2$, and $\gamma=0$ otherwise.
	\end{enumerate}	
\end{theorem}

\begin{proof}
	%First suppose  $n\geq m$. 
 Let $D=\gcd(2m,n)$,  $q = p^{2mn/D}$,  $\alpha = (q-1)/(p^n-1)$, and $\beta = (q-1)/(p^m+1)$. By Theorem \ref{cor-p-rank-Cnm}, $\gamma=\#T-(p^m+p^n+\gcd(p^m+1,p^n-1))$,
	where 
	\begin{equation*}
	T = \left\{ (i,j) \in \Na^2 \mid  0 \leq i\alpha \leq j\beta \leq q-1 \; \mbox{and}\; \dbinom{j\beta}{i\alpha} \not\equiv 0 \pmod{p} \right\}.
	\end{equation*}
	Since $T \subseteq \llbracket0, p^n-1\rrbracket \times \llbracket0, p^m+1\rrbracket$ and $\#\{ (i,j) \in T \mid j \in \{0, p^m+1\}\} = p^n+1$, we will focus on the pairs $(i,j) \in T$ with $j \in \llbracket1, p^m\rrbracket$.
	Considering the base-$p$ expansions 
	\begin{equation*}j-1 = a_{m-1}p^{m-1} + \cdots + a_{1}p + a_{0}  \quad \text{    and   } \quad  i = b_{n-1}p^{n-1} + \cdots + b_{1}p + b_{0},\end{equation*}
	observe that
	\begin{equation}\label{casen>m}
	j\beta = \displaystyle\sum_{a=0}^{\frac{n}{D}-1}\displaystyle\sum_{r=0}^{2m-1} c_rp^{2ma+r}\; \mbox{and}\; i\alpha = \displaystyle\sum_{b=0}^{\frac{2m}{D}-1}\displaystyle\sum_{s=0}^{n-1} b_s p^{nb+s},\end{equation}
	where 
	\begin{equation}\label{17102022-1}
	c_r = \left\{
	\begin{array}{ll}
	p-1-a_r & \mbox{if\;}  r \in \llbracket 0, m-1\rrbracket\\
	a_{r-m} & \mbox{if\;}  r \in \llbracket m, 2m-1\rrbracket.
	\end{array}
	\right.
	\end{equation}
	By Lemma \ref{13102022-1}, $\binom{j\beta}{i\alpha} \not\equiv 0 \pmod{p}$ if and only if for every $s \in  \llbracket0, n-1\rrbracket$, we have 
	\begin{equation*}0 \leq b_s \leq \mymin A_t,\end{equation*} where $0\leq t \leq D-1$,  $t \equiv s \pmod{D}$, and 
	\begin{equation}\label{At-geral}
	A_t = \{c_{D\theta +t} \mid  \theta \in \llbracket0, 2m/D-1\rrbracket\}.
	\end{equation}
	Now  consider  the following.
	
	\begin{itemize}
		\item Case $n/d$ even.
		We have $D=2d$, and then \eqref{At-geral}  reads $A_t = \{c_{2d\theta +t} \mid \theta \in \llbracket0, m/d-1\rrbracket\}$.	
		Clearly  for any $u\in \mathbb{N}$, there exists a unique  $t \in \llbracket0, d-1\rrbracket$ such that  either $u\equiv t \pmod{2d}$ or  $u\equiv t +d \pmod{2d}$. In both cases, we have  $u\equiv t  \pmod{d}$,  and the latter case
		occurs if and only if  $\lfloor  u/d    \rfloor$ is odd. That is,  taking  into account the parity of $\lfloor  u/d    \rfloor$, we may replace reduction modulo $2d$ by  reduction modulo $d$.
		This observation underlies the proof of this case.
		
		Note that for  $t \in  \llbracket0, d-1\rrbracket$,  considering the sets 
		\begin{equation*}
		A_t \cup A_{d+t}=\left\{c_{\theta d+t}     \mid  \theta \in \llbracket 0, 2m / d-1 \rrbracket\right\}
		\end{equation*}
		and the values
		$s \in  \llbracket0, n-1\rrbracket$, with $s \equiv t \pmod{d}$, the condition  $\binom{j\beta}{i\alpha} \not\equiv 0 \pmod{p}$ is equivalent to
		\begin{equation*}
		0\leq b_s \leq \left\{\begin{array}{ll}
		\mymin A_t& \text { if }  \lfloor  s/d    \rfloor \text { is even} \\
		\mymin A_{d+t}& \text { if }  \lfloor  s/d    \rfloor \text { is odd}. \\
		\end{array} \right.
		\end{equation*}
		Therefore,  Lemma 4.4 gives $\#T = \#L+p^n+1$, where 
		%{\footnotesize
		\begin{equation*}%\label{L-main}
		L=\left\{ \begin{array}{l|l}\left(a_0, \ldots, a_{m-1}, b_0, \ldots, b_{n-1}\right) &  \begin{array}{l}a_r  \in  \llbracket 0, p-1 \rrbracket, \\ 0 \leq b_s \leq \mymin  A_t \text { if }  \lfloor  s/d    \rfloor \text { is even},  \\ 0 \leq b_s \leq \mymin  A_{t+d} \text { if } \lfloor  s/d    \rfloor \text { is odd}, \\
		t \in \llbracket 0, d-1 \rrbracket, \text { and } s\equiv t \pmod{d} \end{array}\end{array}\right\}.
		\end{equation*}%}
		The two sets $A_t$ and $A_{d+t}$ are given in terms of the same elements $a_{\theta d+t}$,  although in different ways.  Indeed
		for $c_{d \theta+t} \in A_t \cup A_{t+d}$, equation \eqref{17102022-1} yields
		\begin{equation}\label{aux-1}
		c_{d \theta+t}=\left\{\begin{array}{ll}
		p-1-a_{d \theta+t} & \text { if } \theta\leq m/d-1\\
		a_{d(\theta-m/d)+t} & \text { if } \theta \geq m/d.
		\end{array} \right.
		\end{equation}
		Since $m/d$ is odd, collecting the elements  $c_{d \theta+t}$  in \eqref{aux-1} according to the parity of $\theta$ gives
		%{\footnotesize
		\begin{equation*}A_t=\left\{p-1-a_{d \theta_0+t}, \, a_{d \theta_1+t} \middle\vert \begin{array}{l}\theta_i \in \llbracket 0, m / d-1 \rrbracket, \\ \theta_i\equiv i \pmod{2} \end{array}\right\}=\left\{p-1-e_{\theta d+t} \mid \theta \in \llbracket 0, m / d-1 \rrbracket \right\}\end{equation*}%}
		and 
		%{\footnotesize
		\begin{equation*}A_{d+t}=\left\{p-1-a_{d \theta_0+t}, \, a_{d \theta_1+t} \middle\vert \begin{array}{l}\theta_i \in \llbracket 0, m / d-1 \rrbracket, \\ \theta_i\equiv i \pmod{2} \end{array}\right\}  = \left\{e_{\theta d+t}  \mid  \theta \in \llbracket 0, m / d-1 \rrbracket\right\},\end{equation*}%}
		where 
		$e_{\theta d+t}:=\left\{\begin{array}{ll}
		a_{d\theta+t} & \text { if } \theta  \text { is even }\\
		p-1-a_{d \theta+t} & \text { if } \theta  \text { is odd }
		\end{array} \right.,
		$
		with  $\theta \in \llbracket 0, m / d-1 \rrbracket$.

		Now for  a fixed $t \in \llbracket 0, d-1 \rrbracket $ and $u = (e_0, \ldots, e_{m-1}, b_0, \ldots, b_{n-1}) \in L$, we consider  the set of  indices $\{{r_1}, \ldots, {r_{m/d}}, {s_1}, \ldots, {s_{n/d}}\}$, whose corresponding entries in $u$ are such that 
		\begin{equation*}
		0\leq b_{s_\lambda} \leq \left\{\begin{array}{ll}
		\mymin  \{e_{r_\mu}:\, \mu \in \llbracket 1, m/d\rrbracket\} & \text { if }   \lambda \in \llbracket 1, n/(2d)\rrbracket\\
		\mymin \{p-1-e_{r_\mu}: \mu \in \llbracket 1, m/d\rrbracket\} & \text { if }   \lambda \in \llbracket n/(2d)+1, n/d\rrbracket
		\end{array} \right.,
		\end{equation*}
		where  $r_\mu \equiv s_\lambda \equiv t  \pmod{d}$ for $(\mu,\lambda) \in \llbracket 1, m/d\rrbracket \times \llbracket 1, m/d\rrbracket$, and $\lfloor  s_\lambda/d    \rfloor$  is odd  for  $\lambda \leq n/2d$, and even otherwise. Thus $L$ can be seen as the cartesian product of the  $d$ sets 
		%{\footnotesize
		\begin{eqnarray}\label{Lt-2} 
		L_t = \left\{(e_{r_1}, \ldots, e_{r_{m/d}}, b_{s_1}, \ldots, b_{s_{n/d}}) \middle\vert  \begin{array}{l}
		e_{r_\mu} \in \llbracket 0, p-1\rrbracket,\\ 
		0 \leq  b_{s_\lambda} \leq \displaystyle\min\{e_{r_\mu} \mid \mu \in \llbracket1, m/d\rrbracket\},\; \lambda \leq \frac{n}{2d},\\ 
		0 \leq b_{s_\lambda} \leq \displaystyle\min\{p-1-e_{r_\mu} \mid \mu \in \llbracket1, m/d\rrbracket\},\; \lambda > \frac{n}{2d},\\
		s_\lambda \equiv r_\mu \equiv t \pmod{d}
		\end{array}  \right\}.
		\end{eqnarray}%}
		
		\item Case $n/d$ odd.  Now we have $D=d$, which simplifies  the analysis.   Indeed equations 
		\eqref{17102022-1} and \eqref{At-geral} immediately yield 
		\begin{equation*}A_t = \{a_{\theta d+t},\,  p-1-a_{\theta d+t}  \mid  \theta \in \llbracket0, m/d-1\rrbracket\},\end{equation*}
		and then  $\#T = \#L+p^n+1$, where 
		%{\footnotesize
		\begin{eqnarray*}
			L= \left\{(a_0, \ldots, a_{m-1}, b_0, \ldots, b_{n-1}) \middle\vert  \begin{array}{l}
				a_{r} \in \llbracket 0, p-1\rrbracket,\\ 
				0 \leq  b_{s} \leq \mymin  \{a_{\theta d+t},\, p-1-a_{\theta d+t} \mid \theta \in \llbracket0, m/d-1\rrbracket\},\\ 
				t \in \llbracket 0, d-1 \rrbracket,\text { and } s\equiv r \equiv t \pmod{d}
			\end{array}\right\}.
		\end{eqnarray*}%}
		As in the first case, 	 $L$ can be seen as the cartesian product of the  $d$ sets 
		%{\tiny
		\begin{eqnarray}\label{Lt-1}
		L_t = \left\{(a_{r_1}, \ldots, a_{r_{m/d}}, b_{s_1}, \ldots, b_{s_{n/d}}) \middle\vert  \begin{array}{l}
		a_{r_i} \in \llbracket 0, p-1\rrbracket, \\
		0 \leq b_{s_j} \leq \displaystyle\min\{a_{r_i}, p-1-a_{r_i}\mid  i \in \llbracket1, m/d\rrbracket\}\\
		s\equiv r \equiv t \pmod{d}
		\end{array}
		\right\}.
		\end{eqnarray}%}
	\end{itemize}
	Independent of $t$, the cardinalities of the sets $L_t$ in \eqref{Lt-2} and	  \eqref{Lt-1} are given by Lemmas \ref{+formulas-lema3} and \ref{+formulas-lema2.1}, respectively.
	Using  Lemma \ref{04112022-1} and $\#T = (\#L_t)^d+p^n+1$ for each of the two cases,  the result follows.
	
	%	Now suppose that $n < m$. Clearly $\mathcal{F}$ is isomorphic to $y^{p^n-1}=x^{p^m+1}+1$, and as in previous case, the $p$-rank of $\mathcal{F}$ is  given by $\gamma=\#\overline{T}-(p^m+p^n+\gcd(p^m+1,p^n-1))$,
%	where 
%	\begin{equation*}
%	\overline{T} = \left\{ (j,i) \in \Na^2 \mid  0 \leq j\beta \leq i\alpha \leq q-1 \; \mbox{and}\; \dbinom{i\alpha}{j\beta} \not\equiv 0 \pmod{p} \right\}.
%	\end{equation*}
%	Recall that $i\alpha$  and  $j\beta$ can be written  as in \eqref{casen>m}. Thus
%	the last assertion in  Lemma \ref{13102022-1} implies that $\#\overline{T}$ is still  given
%	by the same expression, in terms of $n,m,$ and $p$, obtained for $\#T$. Therefore, the result follows.
%	

\end{proof}

\begin{corollary}\label{16112022-1}
	Let $m, n \in \Za_{>0}$ be divisors of $2^{u}+1$ and $2^{v}-1$, respectively. If $v / \gcd(u,v)$ is odd, then $\mathcal{F}: y^{m} = x^{n} + 1$ has zero $2$-rank.
\end{corollary} 
\begin{proof}
	As in the proof of Corollary \ref{+formulas-thm1-cor1}, it suffices to show that $\gamma(\mathcal{F}_{2^{u}+1,2^{v}-1})=0$, which  this follows from (ii) in Theorem \ref{+formulas-thm2}. 
\end{proof}

\section{Formulas for the p-rank of  $y^2=x^n+1$}\label{sec4}

Throughout this section, $p$ is an odd prime, $n$ is a positive integer coprime to $p$, and $\gamma(\Cc_n)$ is the $p$-rank of the curve $\Cc_n: y^2 = x^n + 1$ defined over $\Fa_q$, where  $q=p^h$ and  $h$ is the order of $p$ in $(\Za/n\Za)^{\times}$.  Recall that Corollary \ref{p-rank-c1a} gives the $p$-rank of the curve $\Cc_n$ in terms of the cardinality of
\begin{equation*}S = \left\{i \in \llbracket0, \lfloor n / 2\rfloor\rrbracket \mid  \dbinom{\frac{q-1}{2}}{i \alpha} \not\equiv 0 \pmod{p} \right\},\end{equation*}
where $\alpha = (q-1)/n$. Since $\lfloor n / 2\rfloor\alpha$ is the largest multiple of $\alpha$ in  $\llbracket0,(q-1)/2\rrbracket$, Lucas' Lemma yields the following version of Corollary \ref{p-rank-c1a}.

\begin{remark}\label{lucas-help}
	If $M$ is the  number of values $u=u_1p^{h-1}+\cdots+u_{h-1}p+u_h$, with $u_i \in \llbracket0,(p-1)/2\rrbracket$, that are  divisible by $\alpha=(q-1)/n$, then 
	\begin{equation}\label{p-adic count}
	\gamma(\Cc_n) = \left\{
	\begin{array}{ll}
	M - 1 & \mbox{\;if\;}  n  \mbox{\;is odd\;} \vspace{0.2cm} \\
	M - 2& \mbox{\;if\;}  n  \mbox{\;is even\;}.
	\end{array}
	\right.
	\end{equation}
	
\end{remark}

The following provides some  consequences of the results in the previous section.
\begin{corollary}\label{corollary-p-rank-c1a}
	If $r \in \Za_{>0}$, then
	\begin{enumerate}[\rm(i)]
		\item $\gamma(\Cc_{p^r-1}) = \big( (p+1)/2\big)^r-2$
		\item $\gamma(\Cc_{2(p^r+1)}) = \big( (p+1)/2\big)^r$
		\item $\gamma(\Cc_{2(p^r-1)}) = \big( (p+1)/2\big)^r-2$  
		\item $\gamma(\Cc_{p^{2r}+p^r+1}) = \big( (p+3)(p+1)/8\big)^r$
		\item $\gamma(\Cc_{2(p^{2r}+p^r+1)}) = 2\big( (p+3)(p+1)/8\big)^r$
		\item  $\gamma(\Cc_{p^{3r}-p^{2r}+p^r-1})=\big( (p^2+2p+3)(p+1)/12\big)^r-2$.
	\end{enumerate}
\end{corollary}
\begin{proof}
	Assertion (i) follows directly from \eqref{c-m,p^h-1 eq 2} in Theorem \ref{c-m,p^h-1}. Assertion 
	(v) follows   from (iv) together with $\gamma(\Cc_{2(p^{2r}+p^r+1)}) = 2\gamma(\Cc_{p^{2r}+p^r+1})$, given by 
	item (iv) in Corollary \ref{cor-04022020}.
	The proofs of  (ii) and (iii) are simpler versions of the proofs of (iv) and (vi), respectively. Thus  we present proofs for (iv) and (vi) only. 
	Regarding (iv), note that $q=p^{3r}$ and $\alpha=(p^{3r}-1)/(p^{2r}+p^r+1)=p^r-1$.
	From Remark \ref{lucas-help}, it suffices to count the values 
	\begin{equation}\label{u-values}
	u = \sum\limits_{i=1}^{r}a_ip^{3r-i}   +\sum\limits_{i=1}^{r}b_ip^{2r-i} +\sum\limits_{i=1}^{r}c_ip^{r-i}, 
	\end{equation}
	with $a_i , b_i,c_i \in \llbracket0,  (p-1)/2\rrbracket$,  that are divisible by $p^r-1$. Since
	\begin{equation*}u \equiv \sum\limits_{i=1}^{r}(a_i+b_i+c_i)p^{r-i}  \pmod{p^r-1},\end{equation*}
	it follows that  $p^r-1$ divides $u$ if and only if  it divides $u_0:=\sum\limits_{i=1}^{r}(a_i+b_i+c_i)p^{r-i}$.
	Note that the condition  $a_i , b_i,c_i \in  \llbracket0,  (p-1)/2\rrbracket$  implies   $u_0 \in  \llbracket0,3(p^r-1)/2\rrbracket$, and then
	$p^r-1$ divides $u_0$ if and only if  $u=0$ or $\sum\limits_{i=1}^{r}(a_i+b_i+c_i)p^{r-i}=p^r-1$.
	Thus we are left with the problem of  counting  $a_i , b_i,c_i \in  \llbracket0,  (p-1)/2\rrbracket$ for which  $\sum\limits_{i=1}^{r}(a_i+b_i+c_i)p^{r-i}=p^r-1$. The uniqueness of the base-$p$  expansion 
	implies that the former equality is equivalent to $a_i+b_i+c_i=p-1$,  given that  $a_i+b_i+c_i\leq 3(p-1)/2\leq 2(p-1)$. Since  $c_i = p-1 - (a_i+b_i)$, we only need to count    the values $a_i,b_i \in  \llbracket0,  (p-1)/2\rrbracket$ such that $a_i+b_i \geq (p-1)/2$. Therefore, the number of values in \eqref{u-values} is given by
	\begin{equation*}M=1+\Big( \sum_{a_i=0}^{(p-1)/2}((p+1)/2-a_i) \Big)^r =1+ \Big( \frac{(p+1)(p+3)}{8}\Big)^r,\end{equation*}
	and the result follows from \eqref{p-adic count} as  $n=p^{2r}+p^r+1$ is odd. The proof of (vi) is similar to that of (iv). Indeed,
	for item (vi) we have $q=p^{4r}$ and $\alpha=p^r+1$, and the same argument leads to count $a_i , b_i,c_i,d_i \in  \llbracket0,  (p-1)/2\rrbracket$ 
	for which
	\begin{equation*}u_0:=|\sum\limits_{i=1}^{r}(-a_i+b_i-c_i+d_i)p^{r-i}| \leq \sum\limits_{i=1}^{r}|-a_i+b_i-c_i+d_i|p^{r-i}    \leq \sum\limits_{i=1}^{r}(p-1)p^{r-i} =p^r-1\end{equation*}
	is divisible by $p^r+1$. Clearly, this condition is equivalent to $u_0=0$,  that is, $a_i+c_i=b_i+d_i$ for all $i\in \llbracket1,r\rrbracket$. Note that for each $k \in \llbracket0,  p-1\rrbracket$, the number of solutions
	$(x,y) \in  \llbracket0,  (p-1)/2\rrbracket^2$   to $x+y=k$ is  $k+1$ if $k\leq (p-1)/2$, and $p-k$ otherwise.
	Therefore, the number of values  of $u$,  in the notation of Remark \eqref{lucas-help}, is given by
	\begin{equation*}M=\Big( ((p+1)/2)^2 + 2\sum_{i=0}^{(p-1)/2}   i^2  \Big)^r = \Big( \frac{(p^2+2p+3)(p+1)}{12}\Big)^r,\end{equation*}
	and  since $n$ is even, the result follows.	
\end{proof}

In the remainder of this section, we incorporate an additional technique to provide a formula for the p-rank $	\gamma(\Cc_n)$ given in \eqref{p-adic count}. This will come from  observing that 
%$Note that  As before, we will be interested in counting the values
%Let $n \in \Za_{>0}$, and let $\alpha = (p^h-1)/n$, where $h$ is the order of $p$ in $(\Za/n\Za)^{\times}$. By Corollary %\ref{p-rank-c1a} and Lucas' Lemma \ref{lucas' theorem}, to compute the $p$-rank of $\Cc_n$, it is enough to count the values of
the problem of counting the values
\begin{equation*}u = x_1 + x_2p + \cdots + x_hp^{h-1},\end{equation*}
with $x_i \in \llbracket0, (p-1)/2\rrbracket$, that are divisible by  $\alpha=(q-1)/n$
can be  framed  in terms of the problem of finding the number of solutions $(x_1,  \ldots, x_h) \in \llbracket0, \alpha t+s\rrbracket^h$ to
\begin{equation*}a_1X_1 + \cdots + a_hX_h \equiv 0 \pmod{\alpha},\end{equation*}
where $a_i \in (\Za/\alpha\Za)^{\times}$, and  $(p-1)/2= t\alpha+s$ for unique integers $t \in \mathbb{N}$ and $s \in \llbracket0, \alpha-1\rrbracket$. In this regard, the following presents some facts that will be used in the subsequent proofs.

\begin{lemma}\label{number-of-solutions}
	Let $h, t, \alpha \in \Za_{>0}$ and  $a_1, \ldots, a_h \in (\Za/\alpha\Za)^{\times}$. If $b, s   \in \llbracket0, \alpha-1\rrbracket$, then the number $N$ of solutions $(x_1, \ldots, x_h) \in  \llbracket0,  t\alpha+s\rrbracket^h$ to 
	\begin{equation}\label{number-of-solutions-11-07-2022}
	a_1X_1  + \cdots + a_h X_h \equiv b \pmod{\alpha}
	\end{equation}
	is given by
	\begin{equation}\label{N-value}
	N = \dfrac{(t\alpha+s+1)^h-(s+1)^h}{\alpha} + \delta,
	\end{equation}
	where $\delta     \in \llbracket0, (s+1)^{h-1}\rrbracket$   is the number of solutions $(y_1, \ldots, y_h) \in \llbracket0,  s\rrbracket^h$ to \eqref{number-of-solutions-11-07-2022}. Morevover,
	\begin{enumerate}[\rm(i)]
		\item if $s=0$, then $\delta = 1$ when $b=0$, and $\delta = 0$ otherwise
		\item if $s=\alpha -1$, then $\delta = \alpha^{h-1}$
		\item if $s = \alpha-2$, then 
		\begin{eqnarray*}
			\delta = \left\{	\begin{array}{ll}
				\dfrac{(\alpha-1)^h-(-1)^h}{\alpha}+ (-1)^h & \mbox{if\;}  \displaystyle\sum_{i=1}^h a_i \equiv -b \pmod{\alpha} \vspace{0.3cm} \\
				\dfrac{(\alpha-1)^h-(-1)^h}{\alpha} & \mbox{otherwise}				
			\end{array} \right.
		\end{eqnarray*}
		\item if $s=\alpha/2$,  $b=0$, and $h$ is odd, then
		$\delta = \dfrac{(\alpha/2+1)^h-1}{\alpha}+\frac{1}{2}.$
	\end{enumerate} 
	
\end{lemma}

\begin{proof}
	For each $h \in \Za_{>0}$ and $b \in \llbracket0, \alpha-1\rrbracket$, let $N_{h,b}$ and $\delta(h,b)$ denote  the number of solutions to
	\begin{equation*}%\label{number-of-solutions-11-07-2022-1}
	a_1X_1  + \cdots + a_h X_h \equiv b \pmod{\alpha}
	\end{equation*}
	in  $\llbracket0, t\alpha+s\rrbracket^h$  and  $\llbracket0, s\rrbracket^h$, respectively.  	We proceed by induction on $h$  to prove that
	\begin{equation*}
	N_{h,b} = \dfrac{(t \alpha+s+1)^h-(s+1)^h}{\alpha} + \delta(h,b).
	\end{equation*}
	Clearly,  $N_{1,b}=t+\delta(1,b)$, which  gives the case  $h=1$.
	Let  us  assume the result for $h-1\geq 1$. For each $j \in \llbracket0, t\alpha+s\rrbracket$, let $N_{h,b}^j$ be the number of solutions  $(x_1,\ldots,x_{h-1})\in \llbracket0, t\alpha+s\rrbracket^{h-1}$  to
	\begin{equation*}%\label{22092022-5.4}
	a_1X_1 + \cdots +  a_{h-1}X_{h-1} +  a_hj \equiv b \pmod{\alpha}.
	\end{equation*}
	Note that  $N_{h,b}^j = N_{h-1,b_j}$, where $b_j \in   \llbracket0,  \alpha-1\rrbracket$   and $b_j \equiv b-a_hj \pmod{\alpha}$, and  
	\begin{equation*}N_{h,b} =  \displaystyle\sum_{j=0}^{t\alpha+s} N_{h,b}^j  = t \displaystyle\sum_{j=0}^{\alpha-1} N_{h-1,j} + \sum_{j=0}^s N_{h-1,b_j} = t(t\alpha+s+1)^{h-1} + \sum_{j=0}^s N_{h-1,b_j}.\end{equation*}
	From the induction  hypothesis,
	\begin{eqnarray*}
		\sum_{j=0}^s N_{h-1,b_j} &=& (s+1)\cdot   \dfrac{(t \alpha+s+1)^{h-1}-(s+1)^{h-1}}{\alpha} + \sum_{j=0}^s \delta(h-1,b_j)\\
		&=&  \dfrac{ (s+1)(t \alpha+s+1)^{h-1}-(s+1)^{h}}{\alpha} + \delta(h,b),\\
	\end{eqnarray*}
	and then  $N_{h,b} = \dfrac{(t\alpha+s+1)^h-(s+1)^h}{\alpha} + \delta(h,b)$, which  completes the proof.  In addition, by the definition of $\delta(h,b)$, one can readily see that $\delta(h,b) \in   \llbracket0,  (s+1)^{h-1}\rrbracket$.
	Assertions (i) and (ii) are straightforward, and the proof of (iii) follows  by induction, as in the proof of \eqref{N-value}. To prove assertion (iv), let $M$ be the set of the solutions to 
	\begin{equation*}
	a_1X_1 +  \cdots + a_hX_h \equiv s \pmod{2s}
	\end{equation*}
	in   $\llbracket0, s\rrbracket^h$. Since $b=0$ and $h$ is odd, we have that   $(x_1, \ldots, x_h)\in \llbracket0, s\rrbracket^h$  is a solution to \eqref{number-of-solutions-11-07-2022} if and only if $(s-x_1, \ldots, s-x_h) \in M$. This implies  $\delta = \#M$.
	Note that  $\delta+\#M$ is the number of solutions $(x_1, \ldots, x_h)\in \llbracket0, s\rrbracket^h$ to 
	\begin{equation*}
	a_1X_1 + \cdots + a_hX_h \equiv 0 \pmod{s},
	\end{equation*}
	and then  assertion  (i) gives  
	$2\delta = \delta+\#M = \frac{(s+1)^{h}-1}{s}+1$, that is, $\delta = \frac{(\alpha/2+1)^h-1}{\alpha} + \frac{1}{2}$. 
	
\end{proof}

\begin{theorem}\label{p-rank-28-06-2022}
	Let $n \in \Za_{>0}$ and  $\alpha = (p^h-1)/n$, where $h$ is the order of $p$ in $(\Za/n\Za)^{\times}$. If $s \in \llbracket0, \alpha-1\rrbracket$ is such that $s \equiv (p-1)/2 \pmod{\alpha}$, then  
	\begin{equation*}
	\gamma(\Cc_n) = \dfrac{n}{p^h-1}\Big(\Big(\dfrac{p+1}{2}\Big)^h - \delta_n\Big),\; \mbox{where}\; \delta_n = \left\{
	\begin{array}{ll}
	(s+1)^h - \alpha(\delta-1) & \mbox{\;if\;}  n  \mbox{\;is odd\;} \vspace{0.2cm} \\
	(s+1)^h - \alpha(\delta-2)   & \mbox{\;if\;}  n  \mbox{\;is even\;},
	\end{array}
	\right.
	\end{equation*}
	and $\delta\in \llbracket1, (s+1)^{h-1}\rrbracket$ is the number of solutions $(x_1, \ldots, x_h) \in \llbracket0,s\rrbracket^h$ to 
	\begin{equation}\label{21092022-2}
	X_1 + pX_2 +  \cdots + p^{h-1}X_h \equiv 0 \pmod{\alpha}.
	\end{equation}
	Moreover, 
	\begin{enumerate}[\rm(i)]
		\item if $\alpha \mid (p-1)/2$, then $n$ is even and $\delta_n = \alpha + 1$
		\item if $\alpha \mid (p+1)/2$, then $\delta_n = \left\{
		\begin{array}{ll}
		\alpha & \mbox{\;if\;}  n  \mbox{\;is odd\;} \vspace{0.2cm} \\
		2\alpha  & \mbox{\;if\;}  n  \mbox{\;is even\;}
		\end{array}
		\right.		$
		\item if $\alpha \mid (p+3)/2$ and $p-1 \nmid n$, then $\delta_n = \left\{
		\begin{array}{ll}
		\alpha + (-1)^h  & \mbox{\;if\;}  n  \mbox{\;is odd\;} \vspace{0.2cm} \\
		2\alpha + (-1)^h  & \mbox{\;if\;}  n  \mbox{\;is even\;}
		\end{array}
		\right.$
		\item if $\alpha \mid (p+3)/2$ and $p-1 \mid n$, then $n$ is even and $\delta_n = 2\alpha + (-1)^h(1-\alpha)$
		\item if $n$ is odd and $\alpha \mid (p-1)$, then $h$ is odd and $\delta_n = (\alpha+2)/2$.
	\end{enumerate}
\end{theorem}
\begin{proof}
	It follows from Remark \ref{lucas-help} that 
	\begin{equation*}
	\gamma(\Cc_n) = \left\{
	\begin{array}{ll}
	M - 1 & \mbox{\;if\;}  n  \mbox{\;is odd\;} \vspace{0.2cm} \\
	M - 2& \mbox{\;if\;}  n  \mbox{\;is even\;},
	\end{array}
	\right.
	\end{equation*}
	where $M$ is the number of 
	%values $u=u_1p^{h-1}+\cdots+u_{h-1}p+u_h$, with $u_i \in \llbracket0,(p-1)/2\rrbracket$, that are be divisible by $\alpha$, that %is, $M$ is the number of 
	solutions $(x_1, \ldots, x_h) \in \llbracket0,(p-1)/2\rrbracket^h$ to \eqref{21092022-2}. Let $t \in \Na$ and $s \in \llbracket0, \alpha-1\rrbracket$ be such that $(p-1)/2=\alpha t + s$. 
	As $p^i \in (\Za/\alpha\Za)^{\times}$ for all $i \in \llbracket0, h-1\rrbracket$,  Lemma \ref{number-of-solutions} gives 
	\begin{equation*}
	M = \dfrac{(t\alpha + s +1)^h-(s+1)^h}{\alpha} + \delta = \dfrac{1}{\alpha}\Big(\Big(\dfrac{p+1}{2}\Big)^h-\big((s+1)^h-\alpha\delta\big)\Big).
	\end{equation*}
	Therefore, $\gamma(\Cc_n) = \dfrac{n}{p^h-1}\Big(\Big(\dfrac{p+1}{2}\Big)^h - \delta_n\Big)$.  
	The proofs of  assertions (i) --- (v) are all similar and simple. We provide  a proof for assertion (iii) only. First note that the conditions 
	\begin{equation*}
	s \equiv (p-1)/2 \pmod{\alpha}, \,  s \in \llbracket0, \alpha-1\rrbracket,  \text{  and  } (p+3)/2 \equiv 0 \pmod{\alpha}
	\end{equation*}
	give $s=\alpha-2$, which leads  to item (iii) of Lemma \ref{number-of-solutions}. Bearing \eqref{number-of-solutions-11-07-2022} and \eqref{21092022-2} in mind, note that $a_i := p^{i-1} \in (\Za/\alpha\Za)^{\times}$ for all $i \in \llbracket1, h\rrbracket$, and 
	\begin{equation*}
	\sum_{i=1}^{h} a_i = \sum_{i=0}^{h-1} p^i = \frac{p^h-1}{p-1} = \alpha \left( \frac{n}{p-1} \right) \not\equiv 0 \pmod{\alpha},
	\end{equation*}
	as $p-1 \nmid n$.  Hence Lemma \ref{number-of-solutions} gives $\delta = ((\alpha-1)^h-(-1)^h)/\alpha$, and then 
	\begin{equation*}
	\delta_n = (\alpha -1)^h - \alpha (\delta - \gcd(n,2)) = \left\{
	\begin{array}{ll}
	\alpha + (-1)^h  & \mbox{\;if\;}  n  \mbox{\;is odd\;} \vspace{0.2cm} \\
	2\alpha + (-1)^h  & \mbox{\;if\;}  n  \mbox{\;is even\;}.
	\end{array}
	\right.
	\end{equation*}
\end{proof}

%Now using Theorem \ref{p-rank-28-06-2022}, we are able to compute the $p$-rank of the curves in Table \ref{tab-Cn2}.

\section{Explicit formulas for other types of curves}\label{sec5}
In the prior  sections, we considered an array  of curves $ y^{m} = x^{n} + 1$ and  showed 
how Theorem \ref{teo-intro} leads to explicit  formulas of their $p$-ranks. In this final section, we further discuss
how  the previous  results can be used to obtain the $p$-rank of other types of curves.
%Clearly,  the  $p$-rank formula for several other curves of this type

%have shown that Theorem \ref{teo-intro} easily leads to formulas for the $p$-rank  of several curves . One can certainly
As mentioned in the Introduction,  in  \cite[Remark 4]{KGT}, the authors computed the $p$-rank of the Dickson-Guralnick-Zieve curve 
\begin{equation*}
\Cc: \dfrac{(x^{q^3}-x)(y^q-y)-(x^q-x)(y^{q^3}-y)}{(x^{q^2}-x)(y^q-y)-(x^q-x)(y^{q^2}-y)} = 0
\end{equation*}
in terms of the $p$-rank of  $\mathcal{F}_{q-1}: y^{q-1}=x^{q-1}+1$. More precisely, they used the Deuring-Shafarevich formula \cite[Theorem 4.2]{Subrao} to show that 
\begin{equation*}
\gamma(\mathcal{C})=2q^4+\left(\gamma\left(\mathcal{F}_{q-1}\right)-4\right)q^3  + q + 1
\end{equation*}
and computed $\gamma(\mathcal{C})$ when $q=p$ is prime. From \cite[Theorem 19]{Bassa}, or its generalization in Theorem \ref{+formulas-thm1}, we have $\gamma(\mathcal{F}_{q-1})=\left(((p+1) / 2)^h-3\right)q+3$, where $q=p^h$, which finally gives the $p$-rank 
\begin{equation*}
\gamma(\Cc) = \big( ((p+1)/2)^h -1\big)q^4-q^3+q+1
\end{equation*}
of the Dickson-Guralnick-Zieve curve in the general case.
Note that  Theorem \ref{+formulas-thm1} and the subsequent results in Section \ref{sec3} also enable us to compute  the $p$-rank of possible generalizations of $\Cc$, arising as $p$-power Galois extensions of $y^{p^m\pm 1}=x^{p^n\pm 1}+1$.

For another example, let us consider the curve  over $\mathbb{F}_q$, where  $q=p^h$ is odd,  given by
\begin{equation*}
\mathcal{X}: (x+y)^{q+1} - 2 (x^qy^q+xy) = 0.
\end{equation*}
The curve $\mathcal{X}$, recently presented and investigated in \cite{BKS}, has a significant role
in the classification of certain curves with a large automorphism group.  Once again, using the  Deuring-Shafarevich formula, it can be shown  that the  p-rank of $\mathcal{X}$ is given by  $\gamma(\mathcal{X})=(\gamma(\Cc_{2(q-1)})+1)q-1$,
where $\Cc_{2(q-1)}$ is the curve $y^2=x^{2(q-1)}+1$  \cite[Proposition 8.12]{BKS}. Now, after Corollary \ref{corollary-p-rank-c1a}, we immediately have 
\begin{equation*}
\gamma(\mathcal{X})=\big( ((p+1)/2)^h -1\big)q-1.
\end{equation*}

In addition to the Deuring-Shafarevich formula, Lemma \ref{p-rank-morphism} and  Theorem \ref{teo-13032020} clearly provide methods  to compute the $p$-rank of other types of curves, based on the p-rank of $y^m=x^n+1$. This was already explored in
Corollaries \ref{p-rank222a}, \ref{+formulas-thm1-cor1} and \ref{16112022-1} and Theorems  \ref{p-rank-Fnm-kani-rosen} and \ref{+formulas-thm4}. But now the results of Sections \ref{sec4} and \ref{sec5} allow  us to  go  beyond. 
For instance, let us consider the curve 
\begin{equation*}
\mathcal{D}_n: y^2 = x(x^n + 1)
\end{equation*}
defined over $\Fa_q$, where $q$ is a power of an odd prime $p$, and $n \in \Za_{>0}$ is not divisible by $p$.  Note that (ii) in Corollary \ref{cor-04022020} gives
\begin{equation}\label{140922}
\gamma(\mathcal{D}_n) = \gamma(\mathcal{C}_{2n}) - \gamma(\mathcal{C}_{n}).
\end{equation}
Let us recall the well-known fact that if $n \mid q+1$, then $\Cc_n$ is $\Fa_{q^2}$-maximal, and that $\mathcal{D}_{n}$ is $\Fa_{q^2}$-maximal if and only if $q \equiv -1 \mbox{\;or\;} n+1 \pmod{2n}$ \cite[Theorem 1]{Saeedx}. In both cases, the curve, being supersingular,  has zero $p$-rank. The following are  immediate consequences of the aforementioned facts.

\begin{corollary}\label{cor1-12032020}
	If $h \in \Za_{>0}$, then $\gamma(\mathcal{D}_{p^h+1}) = \big( (p+1)/2\big)^h$.
\end{corollary}
\begin{proof}
	From \eqref{140922}, we have $\gamma(\mathcal{D}_{p^h+1}) = \gamma(\Cc_{2(p^h+1)}) - \gamma(\Cc_{p^h+1})$. Since  $\Cc_{p^h+1}$ is  supersingular, item (ii) of Corollary \ref{corollary-p-rank-c1a} gives $\gamma(\mathcal{D}_{p^h+1}) = \big( (p+1)/2\big)^h$.
\end{proof}

\begin{corollary}\label{cor-12032020b}
	Let  $h, \alpha \in \Za_{>0}$ be such that $\alpha$ is  an odd divisor of $p^h-1$. Then $\gamma(\Cc_{(p^h-1)/\alpha}) = \gamma(\Cc_{2(p^h-1)/\alpha})$. 
\end{corollary}
\begin{proof}
	This  follows directly from \eqref{140922} and the  supersingularity of the curve 
	$\mathcal{D}_{(p^h-1)/\alpha}$.
\end{proof}

%Furthermore, as a consequence of Theorem \ref{p-rank-28-06-2022}, we have the following result, compiled in Table \ref{tab-Dn}.
Now note that by Corollary \ref{cor-04022020}, if $n$ is odd, then $\gamma(\mathcal{D}_n) = \gamma(\mathcal{C}_n)$. For  $n$ even, the $p$-rank of $\mathcal{D}_n$ is provided in the following result, compiled in Table \ref{tab-Dn}.

\begin{theorem}\label{tab5-24092022}
	Let $n \in \mathbb{Z}_{>0}$ be  even  and  $\alpha = (p^h-1)/n$, where $h$ is the order of $p$ in $(\Za/n\Za)^{\times}$. If $\alpha$ is odd, then $\gamma(\mathcal{D}_n) = 0$. Otherwise,
	\begin{equation*}
	\gamma(\mathcal{D}_n) = \dfrac{n}{p^h-1}\Big(\Big(\dfrac{p+1}{2}\Big)^h - \tilde{\delta}_n\Big),\; \mbox{with}\; \tilde{\delta}_n = 2\delta_{2n}-\delta_{n},
	\end{equation*}
	where $\delta_{n}$ is given as in Theorem \ref{p-rank-28-06-2022}. Moreover, 
	\begin{enumerate}[\rm(i)]
		\item if $\alpha \mid (p-1)/2$, then $\tilde{\delta}_n = 1$
		
		\item if $\alpha \mid (p+1)/2$, then $\tilde{\delta}_n = 0$
		
		\item if $\alpha \mid (p+3)/2$, then
		%and $(p-1)\nmid 2n$ or $(p-1)\mid n$, then $\tilde{\delta}_n = (-1)^h$
		\begin{equation*}
		\tilde{\delta}_n = \left\{\begin{array}{ll}
		(-1)^h(1-\alpha)  & \mbox{\;if\;}  (p-1)\mid 2n  \mbox{\;and\;} (p-1)\nmid n \vspace{0.2cm} \\
		(-1)^h  & \mbox{\;otherwise.\;}
		\end{array}\right.
		\end{equation*}
	\end{enumerate}
\end{theorem}
\begin{proof}
	If $\alpha$ is odd, then $p^h = 2n \big((\alpha-1)/2\big) + n+1$, that is, $\mathcal{D}_{n}$ is $\Fa_{p^{2h}}$-maximal. Therefore, $\gamma(\mathcal{D}_{n})=0$. Suppose that $\alpha$ is even. 
	As $(p^h-1)/(2n)=\alpha/2 \in \Za_{>0}$,  equation \eqref{140922} and Theorem \ref{p-rank-28-06-2022} yield 
	\begin{eqnarray*}
		\gamma(\mathcal{D}_n) &=& \dfrac{2n}{p^h-1}\Big(\Big(\dfrac{p+1}{2}\Big)^h - \delta_{2n}\Big) - \dfrac{n}{p^h-1}\Big(\Big(\dfrac{p+1}{2}\Big)^h - \delta_{n}\Big)\\
		&=& \dfrac{n}{p^h-1}\Big(\Big(\dfrac{p+1}{2}\Big)^h - \tilde{\delta}_n\Big).
	\end{eqnarray*}
	%	The proofs of the assertions (i) --- (iii) are very  similar. We choose to provide the proof of the first part of  (iii) to represent them all. Since $(p^h-1)/(2n)=\alpha/2$, $(p-1)\mid 2n$ and $(p-1)\nmid n$, it follows from items (iv) and (iii) of Theorem \ref{p-rank-28-06-2022} that $\delta_{2n} = \alpha+(-1)^h(2-\alpha)/2$ and $\delta_n = 2\alpha + (-1)^h$. Therefore,
	%	\begin{equation*}
	%	\tilde{\delta}_n = 2\delta_{2n}-\delta_{n} = (-1)^h(1-\alpha).
	%	\end{equation*}
	The proofs of the assertions (i) --- (iii) are quite   similar. We provide the proof  of  (iii) only. Since $(p^h-1)/(2n)=\alpha/2$, it follows from items (iii) and (iv) of Theorem \ref{p-rank-28-06-2022} that 
	\begin{equation*}
	\delta_{2n} = \left\{ \begin{array}{ll}
	\alpha + (-1)^h & \mbox{\;if\;} (p-1) \nmid 2n \\
	\alpha + (-1)^h(2-\alpha)/2 & \mbox{\;otherwise\;} \\
	\end{array}
	\right. \;\mbox{and}\;\;
	\delta_{n} = \left\{ \begin{array}{ll}
	2\alpha + (-1)^h(1-\alpha) & \mbox{\;if\;} (p-1) \mid n \\
	2\alpha + (-1)^h & \mbox{\;otherwise.\;} \\
	\end{array}
	\right.
	\end{equation*}
	%$\delta_{2n} = \alpha+(-1)^h(2-\alpha)/2$ and $\delta_n = 2\alpha + (-1)^h$. 
	Therefore,
	\begin{equation*}
	\tilde{\delta}_n = 2\delta_{2n}-\delta_{n} = \left\{\begin{array}{ll}
	(-1)^h(1-\alpha)  & \mbox{\;if\;}  (p-1)\mid 2n  \mbox{\;and\;} (p-1)\nmid n \vspace{0.2cm} \\
	(-1)^h  & \mbox{\;otherwise.\;}
	\end{array}\right.
	\end{equation*}
\end{proof}

\section{A final remark}
In Sections \ref{sec3} and \ref{sec4}, we used  a few techniques  to show how Theorem \ref{cor-p-rank-Cnm} yields
explicit formulas for the $p$-rank of Fermat-type curves, and in Section \ref{sec5}, we demonstrated  how the results can be used to arrive at the $p$-rank of other curves. 
%The findings already extend several results in the literature,  but  we should note  that they are merely
These sections bring a  number of new results to the literature, but  we should note  that they are merely
illustrative in the sense  that, using  Theorem \ref{cor-p-rank-Cnm},  there is still  room to   attain  $p$-rank formulas for several other curves. For example, with some of the techniques of Section \ref{sec3}, it can be proved that if $m$ and  $n$  are divisors of $p-1$, then the curve $y^{m(p+1)}=x^{n(p+1)}+1$ has $p$-rank 
$$\frac{1}{4}\Big(m n p^2+2\left(2 m n- m- n- d\right) p+3 m n-2 m-2 n-2 d+4\Big)-\frac{1}{12 m n}\left(m^2+n^2+d^2\right)(p-1)^2,$$
where $d=\gcd (m,n)$.
\section{Acknowledgements }
	We would like to thank Rachel Pries for her comments on a previous version of this manuscript, which enabled us to enhance our text.  The first  author was supported by CNPq (Brazil), grant 311572/2019-7.

\appendix
\section{The  cardinality  of a few  sets}

This appendix  provides  the  cardinality  of a few  simple sets that  appear throughout
the proofs of our results in Section \ref{sec3}.

\begin{lemma}\label{20112022-1}
	If $p$ is a prime and
	\begin{equation*}\mathcal{L}=\left\{\left(a, b, c, d\right)  \ \middle\vert  
	\begin{array}{l}
	a,b \in \llbracket1, p\rrbracket,\\
	1 \leq c, d \leq \min \left\{a, b\right\},\\
	a+b \leq c+d +p-1\\
	\end{array}
	\right\},
	\end{equation*}
	then $\#\mathcal{L}=p(p+1)(p^2+p+2)/8.$
\end{lemma}

\begin{proof} For $p=2$, it is straightforward to check that $\#\mathcal{L}=6$. Thus we suppose $p\geq3$, and
	proceed by fixing   $(a,b)  \in \llbracket1, p\rrbracket^2$, and  then counting the pairs $(c,d)$ satisfying the conditions given  in  set $\mathcal{L}$.
	
	Suppose $a<b$. Setting $x:=a-c$ and $y:=a-d$, it suffices to count the pairs $(x,y) \in  \llbracket0, a-1\rrbracket^2$ such that $x+y \leq p-1 -(b-a)$. Note that   
	$(a,b)$ is such that  the square $ \llbracket0, a-1\rrbracket^2$ is below the line $x+y = p-1 -(b-a)$ if and only if   $a+b\leq p+1$. Thus such points give rise to
	$$
	\sum\limits_{b=2}^p\sum\limits_{a=1}^{\min\{  b-1 ,p+1-b \}}a^2=(p^4 + 4p^3 + 2p^2 - 4p - 3)/96
	$$
	elements of $\mathcal{L}$. In the cases $a+b\geq p+2$, we  need to remove the   $(a+b-p-1)(a+b-p)/2$  points above the line $x+y=(p-1)-(b-a)$. That is, the  points $(a,b)$ for which $a+b\geq p+2$ give rise to
	$$
	\sum_{b=\frac{p+3}{2}}^p \sum_{a=p+2-b}^{b-1}(a^2-(a+b-p-1)(a+b-p) / 2)=(5 p^4-4 p^3-2 p^2+4 p-3)/96
	$$
	elements of $\mathcal{L}$. Adding the quantities above gives us the total of  $(p^4 - 1)/16$  elements $(a,b,c,d) \in \mathcal{L}$ for which $a<b$. Clearly, by symmetry, the same holds when $a>b$.
	Hence  $(p^4 - 1)/8$ is the  number of elements $(a,b,c,d) \in \mathcal{L}$ for which $a\neq b$. 
	
	For the cases $a=b$, considering the line $x+y = p-1$, the same procedure gives a total
	of $(2p^3 + 3p^2 + 2p + 1)/8$ elements $(a,b,c,d) \in \mathcal{L}$ for which $a=b$.  Therefore,
	$$\#\mathcal{L}=(p^4 - 1)/8+(2p^3 + 3p^2 + 2p + 1)/8=p(p+1)(p^2+p+2)/8.$$
\end{proof}

\begin{lemma}\label{+formulas-lema1}
	If $b, m, n \in \Za_{>0}$ and
	\begin{equation*}
	\mathcal{L} = \left\{ (a_0, \ldots, a_{m-1},b_0, \ldots, b_{n-1}) \ \middle\vert 
	\begin{array}{l}
	0\leq a_i \leq b-1\\
	0 \leq b_j \leq \min \{a_1, \ldots,a_{m-1}\} 
	\end{array}\right\},
	\end{equation*}
	then $\#\mathcal{L} = \dsum_{i=0}^{b-1} \Big((i+1)^m - i^m\Big)(b-i)^n$.
\end{lemma}

\begin{proof}
	Note that for each $t \in \{0, \ldots, b-1\}$, the number of $m$-tuples $(a_0, \ldots, a_{m-1})$ for which $\min\{a_0, \ldots, a_{m-1}\}=t$ is $(b-t)^m-(b-(t+1))^m$. Since for each such $m$-tuple, we have $(t+1)^n$ possible $n$-tuples $(b_0, \ldots, b_{n-1})$, it follows that
	\begin{equation*}
	\#\mathcal{L} = \dsum_{t=0}^{b-1} \Big(\big(b-t\big)^m-\big(b-t-1\big)^m\Big)(t+1)^n.
	\end{equation*}
	Defining  $i:=b-t-1$, the result follows.
\end{proof}

\begin{lemma}\label{+formulas-lema2.1}
	Let $b, m, n \in \Za_{>0}$. If $b \geq 2$  and
	\begin{equation*}
	\mathcal{L}=\left\{\left(a_0, \ldots, a_{m-1},b_0, \ldots, b_{n-1}\right)\ \middle\vert     \begin{array}{l}  0\leq a_{r}\leq  b-1, \\ 0 \leq b_{s} \leq \displaystyle\min_{r}\{a_r, b-1-a_r\} \end{array}\right\},
	\end{equation*}
	then
	\begin{equation*}
	\#\mathcal{L} =\left\{
	\begin{array}{ll}
	\dfrac{1}{2^n}\Big((b+1)^n+\dsum_{i=1}^{(b-1)/2 } \big((2i+1)^m-(2i-1)^m\big)(b-2i+1)^n\Big) & \mbox{if  $b$ is odd}   \vspace{0.2cm}\\
	2^{m-n}\dsum_{i=1}^{b/2} \big(i^m-(i-1)^m\big)(b-2i+2)^n & \mbox{if  $b$ is even}.
	\end{array}
	\right.
	\end{equation*}
\end{lemma}
\begin{proof}
	Direct inspection shows that  for each $t \in \{0, \ldots, b-1\}$, the number of $m$-tuples $(a_0, \ldots, a_{m-1})$ for which %$\displaystyle\min_{0 \leq r \leq m-1}\{a_r, b-a_r-1\}=i$ is
	$\displaystyle\min_{r}\{a_r, b-a_r-1\}=t$ is
	\begin{equation*}
	\left\{
	\begin{array}{ll}
	\big(b-2t\big)^m-\big(b-2t-2\big)^m & \mbox{if\;} t <(b-1)/2   \vspace{0.2cm}\\
	1 & \mbox{if\;} t =(b-1)/2 \vspace{0.2cm} \\
	0 & \mbox{if\;} t >(b-1)/2.
	\end{array}
	\right.
	\end{equation*}
	For each such $m$-tuple, we have $(t+1)^n$ possible $n$-tuples $(b_0, \ldots, b_{n-1})$. Since $t=(b-1)/2$ occurs if and only if $b$ is odd, it follows that
	\begin{equation*}
	\#\mathcal{L} =\left\{
	\begin{array}{ll}
	\big((b+1)/2\big)^n+\dsum_{t=0}^{(b-3)/2 } \Big(\big(b-2t\big)^m-\big(b-2t-2\big)^m\Big)(t+1)^n & \mbox{if  $b$ is odd}   \vspace{0.2cm}\\
	\dsum_{t=0}^{(b-2)/2} \Big(\big(b-2t\big)^m-\big(b-2t-2\big)^m\Big)(t+1)^n & \mbox{if  $b$ is even}.
	\end{array}
	\right.
	\end{equation*}
	Defining  $i:=(b-2t-1)/2$ when $b$ is odd, and $i:=(b-2t)/2$ when $b$ is even, the result follows.
\end{proof}

\begin{lemma}\label{+formulas-lema3}
	If $b, m, n \in \Za_{>0}$ and 
	\begin{equation*}
	\mathcal{L} =  \left\{(a_0, \ldots, a_{m-1}, b_{0}, \ldots, b_{2n-1}) \middle\vert  \begin{array}{l}
	0 \leq a_r \leq b-1 \\ 
	0 \leq  b_s \leq \displaystyle\min\{a_r: r \in \llbracket0, m-1\rrbracket\},\; s \leq n-1,\\ 
	0 \leq b_s \leq \displaystyle\min\{b-a_r-1: r \in \llbracket0, m-1\rrbracket\},\; s \geq n
	\end{array}  \right\},
	\end{equation*}
	then 
	\begin{equation*}
	\#\mathcal{L} = \dsum_{i=0}^{b-1} \big((b-i)(i+1)\big)^n + \dsum_{1 \leq j \leq i \leq b-1} \big((j+1)^m-2j^m+(j-1)^m\big)\big((b-i)(i-j+1)\big)^n.
	\end{equation*}
\end{lemma}

\begin{proof}
	Note that for each $i \in \llbracket0, b-1\rrbracket$ and $t \in \llbracket0, i\rrbracket$, the number of $m$-tuples $(a_0, \ldots, a_{m-1})$ for which $\max\{a_0, \ldots, a_{m-1}\}=i$ and $\min\{a_0, \ldots, a_{m-1}\}=t$ is $1$ when $t=i$, and $(i-t+1)^m-2(i-t)^m+(i-t-1)^m$ when $t \neq i$. Since for each such $m$-tuple, we have $(b-i)^n(t+1)^n$ possible $2n$-tuples $(b_0, \ldots, b_{2n-1})$, it follows that
	\begin{equation*}
	\#\mathcal{L} = \dsum_{i=0}^{b-1} (b-i)^n(i+1)^n + \dsum_{0 \leq t <i \leq b-1} \big((i-t+1)^m-2(i-t)^m+(i-t-1)^m\big)(b-i)^n(t+1)^n.
	\end{equation*}
	Defining  $j:=i-t$, the result follows.
\end{proof}

\bibliographystyle{cas-model2-names}
\bibliography{references}  %%% Uncomment this line and comment out the ``thebibliography'' section below to use the external .bib file (using bibtex) .





% \end{thebibliography}

\end{document}